\documentclass{article}
\usepackage[colorlinks]{hyperref}
\usepackage{amsmath,amssymb}
\usepackage{physics}
\usepackage{subcaption}
\usepackage{mathtools}
\usepackage{parskip}
\usepackage{graphicx}
\usepackage{comment}
\usepackage{bm}
\usepackage{amsthm}
\usepackage{siunitx}
\usepackage{upgreek}
\usepackage{tikz}
\usetikzlibrary{calc}
\usepackage[toc, page, title]{appendix}
\usepackage{algorithm}
\usepackage{algpseudocode}
\usepackage{authblk}
\usepackage{braket}
\usepackage[pagewise]{lineno}% \linenumbers
\usetikzlibrary{patterns,arrows,decorations.pathreplacing,math,decorations.markings,angles}
\usetikzlibrary{arrows.meta, positioning}
\usepackage[margin=1.0in]{geometry}
\usepackage[backend=biber,style=numeric,sorting=none]{biblatex}
\newcommand{\Kn}{\mathrm{Kn}}
\newcommand{\QBGK}{\mathcal{Q}_{\text{BGK}}}
\newcommand{\QESBGK}{\mathcal{Q}_{\text{ESBGK}}}
\newcommand{\QpESBGK}{\mathcal{Q}_{\text{pESBGK}}}
\newcommand{\Ttensor}{\mathcal{T}}
\newcommand{\Stensor}{\Theta}
\newcommand{\nuESBGK}{\nu_{\text{ES}}}
\newcommand{\nupESBGK}{\nu_{\text{pES}}}
\newcommand{\vmax}{v_{\text{vmax}}}
\newcommand{\normal}{\hat{n}}
\newcommand{\vbar}{\underline{v}}
\DeclareMathOperator{\Div}{div}
\newtheorem{proposition}{Proposition}

\providecommand{\keywords}[1]
{
	\small	
	\textbf{Keywords. } #1
}

\newtheorem{remark}{Remark}

\title{A semi-Lagrangian method for the polyatomic ESBGK model}

\author[1, 2]{Klaas Willems\thanks{Corresponding author: klaas.willems@kuleuven.be}}
\author[3]{Erik Arlemark\thanks{erik.arlemark@asml.com}}
\author[1]{Giovanni Samaey\thanks{giovanni.samaey@kuleuven.be}}
\author[2]{Axel Klar\thanks{axel.klar@rptu.de}}
\affil[1]{Department of Computer Science, KU Leuven, Leuven, Belgium}
\affil[2]{Faculty of Mathematics, RPTU Kaiserslautern-Landau, Kaiserslautern, Germany}
\affil[3]{ASML}

\begin{document}
\date{}
\maketitle

\begin{abstract}
	Polyatomic kinetic models are essential for accurately capturing the thermodynamic behavior of real gases, as internal energy modes significantly influence transport coefficients, relaxation processes, and non-equilibrium effects that cannot be represented by monoatomic models. The polyatomic ESBGK model describes molecular collisions as a relaxation towards a generalized Gaussian distribution with an anisotropic covariance matrix and an exponentially decaying internal energy distribution. We present a new semi-Lagrangian scheme for the polyatomic Ellipsoidal Statistical BGK (ESBGK) model of the Boltzmann equation. The semi-Lagrangian framework, being deterministic and grid-based, removes the time-step restriction associated with the linear transport term by following the method of characteristics. The potentially stiff relaxation term is treated using an implicit A-stable linear multistep method which, owing to the structure of the BGK operator, can be reformulated into a cheap time-stepping scheme. This yields a highly efficient and numerically stable method. The numerical method is asymptotic preserving and stiffly accurate, meaning the scheme asymptotically converges to a scheme for the Euler equations in the vanishing Knudsen limit. In addition, we prove that the first-order scheme, asymptotically converges to the compressible Navier-Stokes equation with correct transport coefficients. Finally, we propose inflow and outflow boundary conditions suitable for BGK-type kinetic equations. We perform simulations of the Fourier and Couette test case to compare the BGK model with Direct Simulation Monte Carlo (DSMC). To conclude, we demonstrate the method on a challenging orifice flow test case with moving boundaries. 
\end{abstract}

\section{Introduction}

Rarefied gas dynamics concerns flows in which non‑equilibrium molecular effects play a significant role. The degree of rarefaction is characterized by the Knudsen number
\[
\Kn = \frac{\lambda}{L},
\]
where \(\lambda\) denotes the mean free path and \(L\) is a characteristic macroscopic length. When \(\Kn \ll 1\), particle collisions are frequent and continuum models such as the Euler or Navier–Stokes–Fourier equations accurately describe the evolution of density, momentum, and energy. For \(\Kn = \mathcal{O}(1)\), however, the continuum hypothesis breaks down and kinetic models become necessary. Such conditions arise naturally in applications including high‑altitude aerodynamics and micro‑electromechanical systems (MEMS) \cite{shen2005}, either due to low pressure (yielding a large mean free path \(\lambda\) ) or due to small geometric scales \(L\).

The Boltzmann equation provides a fundamental statistical description of rarefied gases, but its collision operator contains a five‑fold integral, making direct numerical solution expensive. For this reason, relaxation‑type kinetic models—beginning with the BGK model—have been developed to simplify computations while retaining key physical properties such as conservation laws, entropy dissipation, and correct asymptotic behavior. The ESBGK model extends BGK by enabling correct Prandtl number matching. Since many gases of practical interest are polyatomic, further extensions are required to incorporate internal energy and additional degrees of freedom.

A range of numerical methods exists for kinetic equations. The Direct Simulation Monte Carlo (DSMC) method \cite{bird2013} provides a robust particle‑based approximation of the Boltzmann equation, particularly in high‑Mach or highly non‑equilibrium regimes, but may suffer from statistical noise or high cost for low-speed or transient flows. Deterministic methods, known as discrete velocity methods, provide for accurate alternatives \cite{mieussensDiscreteVelocityModel2000}. They rely on a discretised velocity space, after which the advection and collision operators can be advanced separately using a time‑splitting procedure. For the Boltzmann equation, the collision step can be computed efficiently by exploiting the convolution structure of the collision operator \cite{pareschiFourierSpectralMethod1996}. In contrast, for BGK‑type models, the collision term is typically treated implicitly in time to circumvent the stiffness that arises in the small‑Knudsen‑number regime \cite{pareschiImplicitExplicitRunge2005}. This idea has been explored in the context of finite volume schemes \cite{boscheriHighOrderCentral2020, dechristeCartesianCutCell2016, xuUnifiedGaskineticScheme2010a}.

A particularly well-known deterministic numerical method for BGK-type equations is the so-called semi-Lagrangian method \cite{groppi2014}, which uses the method of characteristics to rewrite the equation as an ODE along characteristics. By tracing characteristics backwards to previous timesteps, a scheme is obtained that is not restricted by a CFL condition. The semi‑Lagrangian framework has been extended to conservative reconstructions \cite{choConservativeSemiLagrangianSchemes2021a, cho2021}, to the ESBGK model \cite{boscarino2025a}, to adaptive velocity grids \cite{brull2014, boscarinoLocalVelocityGrid2022}, and to BGK models for gas mixtures \cite{choConservativeSemiLagrangianSchemes2022}.

\textbf{Contributions.}
The goal of this work is to extend the semi‑Lagrangian methodology to the polyatomic ESBGK model. The main contributions are:

\begin{enumerate}
	\item A new semi‑Lagrangian discretisation of the polyatomic ESBGK model, including an explicit treatment of the collision operator. We follow the procedure from \cite{filbet2011}, in which higher-order moments of the local equilibrium distribution at the next time step are formulated explicitly as a function of moments of a known distribution at the previous time step. In contrast to \cite{filbet2011}, in the polyatomic setting, we obtain a nonlinear system of equations for two temperatures which can be solved using a Newton method. For the velocity covariance matrix, we obtain an explicit expression that simplifies to the case of \cite{filbet2011} for monoatomic particles. 	
	\item A proof that the first-order scheme converges to a consistent discretisation of the Navier–Stokes–Fourier equations with correct transport coefficients. We thereby generalize the proof from \cite{boscarino2025a} for the ESBGK model, to the polyatomic setting. The proof follows a Chapman--Enskog expansion of the numerical scheme, but is more involved due to the presence of additional additional higher-order moments of the distribution function.
	\item An implementation of physically relevant boundary conditions, including diffusive--reflective walls and pressure-driven inflow/outflow conditions. For the diffuse--reflective boundary conditions, we avoid the use of iterative procedures or extrapolation to the boundaries as in \cite{groppiBoundaryConditionsSemiLagrangian2016}. Instead, the boundary conditions are imposed directly after the reconstruction step, prior to the collision step, resulting in a simple and efficient explicit treatment. This approach is then naturally extended to pressure-driven inflow/outflow boundary conditions. To the best of our knowledge, such inflow/outflow boundary conditions for deterministic solvers of BGK-type equations have not yet been reported in the literature.
	\item A comprehensive comparison of BGK, ESBGK, polyatomic ESBGK, DSMC, and Navier–Stokes models for canonical test cases such as Fourier and Couette flow \cite{gallis2011}. This provides a consistent reference framework to assess how these models relate to each other and to benchmark results within the same numerical setting. We thereby aim to complement existing comparisons of polyatomic models using stochastic particle methods for more complex geometries \cite{andries2002, pfeiffer2019}. We also note that preliminary polyatomic Couette flow simulations using a UGKS scheme have recently been reported in \cite{barangerAdaptationUnifiedGasKinetic2026}. 	
	\item Simulations of stationary and moving orifice configurations, demonstrating the performance of the method in non-trivial geometries across continuum and transitional flow regimes. The results include comparisons with Navier--Stokes and DSMC reference data. The moving orifice case further illustrates the ability of the method to handle time-dependent geometries and capture complex flow behaviour such as expansion, recirculation, and shock structures.
\end{enumerate}

The remainder of this paper is organized as follows. Section~\ref{section:Background} presents the mathematical background, describing the Boltzmann equation, BGK, ESBGK, and polyatomic ESBGK models.
Section~\ref{section:SL} introduces the semi‑Lagrangian discretisation. First, we introduce the relevant notation in Section \ref{section:SLNotation}. Then, in Section \ref{section:SLFirstOrder}, we derive the first-order scheme. The implicit time integration of the polyatomic ESBGK collision operator is complicated by its dependence on higher-order moments of the distribution function, such as the internal temperature. To address this, we express these moments in terms of known quantities from previous time steps, which yields nonlinear equations for the two temperatures as well as an explicit expression for the velocity covariance matrix. The extension of the scheme to higher-orders, and a model order reduction technique known as a Chu-reduction are then subsequently discussed in Sections \ref{section:SLSecondOrder} and \ref{section:Chu}. In Section \ref{section:BCs}, we present the diffuse-reflective, and inflow-outflow boundary conditions that can be used for all deterministic methods for BGK-type equations. We then prove in Section \ref{section:AP}, as a generalization of the monoatomic case, that the first-order semi-Lagrangian scheme asymptotically converges to a consistent scheme for the Navier-Stokes equations. Section~\ref{section:Numerics} provides the Fourier and Couette flow test cases, as well as the orifice simulations with moving boundaries. Finally, in section \ref{section:Conclusion}, the conclusion is presented.

\section{Mathematical Background}
\label{section:Background}
This section outlines the kinetic models that form the basis of the polyatomic ESBGK formulation used in this study.

\subsection{The Boltzmann Equation}
\label{section:boltzmann}
The foundational kinetic model for rarefied gases is the Boltzmann equation \cite{cercignaniBoltzmannEquationIts1988}, \begin{align}
	\pdv{f}{t} + v \cdot \nabla_x f = \mathcal{Q}(f), \label{eq:BoltzmannEquation}
\end{align} which governs the evolution of the distribution \(f(t, x, v)\) of monoatomic particles with velocity \(v \in \mathbb{R}^3\), through time \(t \in \mathbb{R}^+\) and space \(x \in \mathbb{R}^3\). The Boltzmann collision operator \(\mathcal{Q}(f)\) is bilinear and contains a five-fold integral. The macroscopic quantities such as the density \(\rho(t, x)\), bulk velocity \(U(t, x)\) and total energy \(E(t, x)\) are then obtained by taking moments of the particle distribution, \begin{align}
	\left( \rho(t, x), \rho(t, x) U(t, x), E(t, x) \right)^\top = \int_{\mathbb{R}^3} \phi(v) f(t, x, v) dv\quad \text{with}\, \phi(v) = \left( 1, v, \frac{1}{2} \abs*{v}^2 \right)^\top. \label{eq:collisionInvariants}
\end{align} The temperature of the gas is related to the total energy in the following way \begin{align}
	E(t, x) = \frac{1}{2} \rho(t, x) \abs{U(t, x)}^2 + \frac{3}{2} R_s T(t, x), \label{eq:totalEnergy}
\end{align} where \(R_s\) is the specific gas constant. The elements of \(\phi(v)\) \eqref{eq:collisionInvariants}, are referred to as collision invariants, because they are conserved by the Boltzmann collision operator in the sense that \begin{align}
	\int_{\mathbb{R}^3} \phi(v) \mathcal{Q}(f) dv = 0. \label{eq:collisionInvariant}
\end{align} Hydrodynamic equations, including the Euler and Navier–Stokes formulations, can be systematically obtained from the Boltzmann equation by applying the Chapman–Enskog expansion \cite{struchtrup2005}. This procedure involves expanding the distribution function in the Knudsen number, \(f = f^0 + \Kn f^1 + \Kn^2 f^2 + \cdots\). At the zero-th order, the Euler equations are obtained. The first-order approximation gives the Navier-Stokes equations \begin{equation}
	\begin{aligned}
		&\pdv{\rho}{t} + \Div_x \left(\rho U\right) = 0, \\
		&\pdv{\rho U}{t} + \Div_x \left(\rho U \otimes U + pI_3\right) = \Kn \Div_x \left[\mu \sigma(u)\right], \\
		&\pdv{E}{t} + \Div_x \left((E + p) U\right) = \Kn \Div_x \left(\mu \sigma(U) U + \kappa \nabla_x T\right),
	\end{aligned} \label{eq:NS}
\end{equation} with the pressure \(p = \rho R_s T\), the viscosity \(\mu = \mu(T)\), the heat conduction \(\kappa\), and the stress tensor \(\sigma(U)\). An important quantity, the Prandtl number, \begin{align}
	\Pr = \frac{5}{2} \frac{R_s \mu}{\kappa}, \label{eq:Pr}
\end{align} quantifies the main mode of diffusion. For \(\Pr \ll 1\), thermal diffusion dominates, whereas for \(\Pr \gg 1\), momentum diffusion dominates. Monoatomic gases typically have a Prandtl number of \(\frac{2}{3}\). An important requirement for kinetic models is their ability to reproduce this physically correct Prandtl number. 
A final property of the Boltzmann equation, is the so-called H-theorem. For space-homogeneous functions, we define the negative entropy \begin{align}
	\mathcal{H}(t) = \int_{\mathbb{R}^3} f \log f dv.
\end{align} It can be shown that solutions to the space-homogeneous Boltzmann equation satisfy the so-called H-theorem \begin{align}
	\dv{\mathcal{H}}{t} = \int_{\mathbb{R}^3} \mathcal{Q}(f) \log f dv \leq 0. \label{eq:HTheorem}
\end{align} Since the quantity \(f \log f\) is bounded from below, and \(\mathcal{H}(t, x)\) must decrease, the Boltzmann equation describes the evolution towards a state of minimum \(\mathcal{H}\), known as thermodynamic equilibrium. At thermodynamic equilibrium, the distribution \(f(t, x, v)\) is the Maxwellian distribution \begin{align}
	\mathcal{M}(t, x, v) = \frac{\rho(t, x)}{\left( 2\pi R_s T(t, x) \right)^{3/2}} \exp \left( -\frac{\abs{v - U(t,x)}^2}{2 R_s T(t, x)} \right).
\end{align}
\subsection{The Classical BGK Model}
\label{section:BGK}
The observation that the distribution \(f\) tends to a Maxwellian distribution is the main motivation for BGK-type models, where the Boltzmann collision operator is replaced by a relaxation operator \begin{align}
	\QBGK(f) = \frac{1}{\tau} \left(\mathcal{M}(t, x, v) - f(t, x, v)\right), \label{eq:ClassicBGK}
\end{align} with \(\tau = \mu(T)/p\) the mean free time. Due to the definition of the density, mean velocity \eqref{eq:collisionInvariants} and temperature \eqref{eq:totalEnergy}, the BGK equation is highly nonlinear. Nevertheless, due of the specific structure of its relaxation operator, the model admits to efficient numerical schemes and retains many essential properties of the Boltzmann equation. Analogous to the Boltzmann equation \eqref{eq:BoltzmannEquation}, the BGK model conserves mass, momentum and energy. Furthermore, a H-theorem for the BGK equation can be proven, \begin{align}
	\dv{\mathcal{H}}{t} &= \dv{}{t} \int_{\mathbb{R}^3} \QBGK(f) \left(\log f + \log \mathcal{M} - \log \mathcal{M} \right) dv 
	= -\dv{}{t} \int_{\mathbb{R}^3} \QBGK(f) \left(\log \mathcal{M} - \log f \right) dv + 0 \leq 0,
\end{align} by adding a zero under the integral and using the conservation of mass property. The main deficit of the BGK model is that it fails to reproduce the correct Prandtl number. Indeed, a Chapman-Enskog expansion of the BGK model to first-order yields a Prandtl number of 1.

\subsection{The ESBGK Model for Monoatomic Gases}
\label{section:ESBGK}
To treat the issue of the BGK model related to the Prandtl number, several solutions exist. In \cite{mieussens2004}, a velocity-dependent collision frequency is considered. Alternatively, a correct Prandtl number can also be obtained by correcting thermodynamic equilibrium in \eqref{eq:ClassicBGK} with a term proportional to the heat flux. This model is known as the Shakhov model \cite{shakhov1972}. However, in this text, we consider the generalized BGK model that was introduced in \cite{holway1965, brull2008}, known as the Ellipsoidal Statistical or Gaussian BGK model (ESBGK). This collision model, which is still a relaxation, replaces the local Maxwellian in the BGK model \eqref{eq:ClassicBGK} with an anisotropic Gaussian, \begin{align}
	\mathcal{G}(t, x, v) = \frac{\rho}{\sqrt{\det(2\pi \Ttensor)}} \exp \left( - \frac{1}{2} \left(v - U(t, x)\right)^\top \Ttensor^{-1} \left(v - U(t, x)\right) \right), 
\end{align} with covariance matrix \(\Ttensor\) and stress tensor \(\Stensor\) \begin{align}
	\Ttensor &= (1 - \nuESBGK) R_s T I_3 + \nuESBGK \Stensor, \\
	\rho(t, x) \Stensor &= \int_{\mathbb{R}^3} (v - U(t, x)) \otimes (v - U(t, x)) f(t, x, v) dv,
\end{align} and free parameter \(-0.5 \leq \nuESBGK < 1\). The collision operator now becomes \begin{align}
	\QESBGK(f) = \frac{1}{\tau (1 - \nuESBGK)} \left(\mathcal{G}(t, x, v) - f(t, x, v)\right). \label{eq:ESBGK}
\end{align} Note that for \(\nuESBGK = 0\), the ESBGK model \eqref{eq:ESBGK} reduces to the BGK model \eqref{eq:ClassicBGK}. The ESBGK model preserves the same collision invariants as the previous models, and the H-theorem was established in \cite{andries2001}. Moreover, a first-order Chapman-Enskog expansion yields a Prandtl number of \(\Pr = \frac{1}{1 - \nuESBGK}\), which for \(\nuESBGK = -0.5\), yields the physically-correct value of \(\Pr = 2/3\). Extensions of both the BGK and ESBGK model to gas mixtures have been developed, see \cite{brull2015, v.bobylev2018} and references therein. All models discussed so far are in principle only valid for monoatomic gases.

\subsection{Polyatomic ESBGK Model}
\label{section:polyatomicESBGK}
Since monoatomic gases constitute only a narrow subset of those encountered in practical applications, more general formulations applicable to polyatomic gases are required. In this work, we therefore consider the polyatomic extension of the ESBGK model from \cite{andriesGaussianBGKModelBoltzmann2000}. In the polyatomic case, the distribution function \(f(t, x, v, I)\) is extended to include the internal energy distribution \(I^{2/\delta}\) of gas particles, where \(\delta\) is the number of additional degrees of freedom. For example, for a diatomic gas, \(\delta = 2\). The collision invariants and associated macroscopic quantities \eqref{eq:collisionInvariants} are now defined as \begin{align}
	\rho(t, x) &= \left< \left< f(t, x, v, I) \right> \right> \\
	\rho(t, x) U(t, x) &= \left< \left< v f(t, x, v, I) \right> \right> \\
	E(t, x) &= \left< \left< \left(\frac{1}{2} \abs{v}^2 + I^{2/\delta} \right) f(t, x, v, I) \right> \right> = \frac{1}{2} \rho(t, x) \abs{U(t, x)}^2 + \rho(t, x) e(t, x), 
\end{align} under the new notation \begin{align}
	\left< \left< \bullet \right> \right> = \int_{v \in \mathbb{R}^3} \int_{I \in \mathbb{R}^+} \bullet \; dv dI. 
\end{align} The specific internal energy \(e(t, x)\) and associated equilibrium temperature \(T_{eq}\) consist of two parts, a translational energy mode, and a internal (rotational) energy mode \begin{equation}
	\begin{alignedat}{2}
		e(t, x)      & = e_{tr}(t, x) + e_{int}(t, x)    &\quad& = \frac{3 + \delta}{2}\, R_s\, T_{eq}(t, x), \\
		e_{tr}(t, x) & = \frac{1}{\rho(t, x)} \Big\langle \Big\langle
		\tfrac{1}{2}\lvert v - U(t,x)\rvert^2 f(t, x, v, I)
		\Big\rangle \Big\rangle             &\quad& = \frac{3}{2}\, R_s\, T_{tr}(t, x), \\
		e_{int}(t, x)& = \frac{1}{\rho(t, x)} \Big\langle \Big\langle
		I^{2/\delta} f(t, x, v, I)
		\Big\rangle \Big\rangle             &\quad& = \frac{\delta}{2}\, R_s\, T_{int}(t, x),
	\end{alignedat}
\end{equation} where the translational temperature \(T_{tr}\) is used to define the pressure \(p = \rho R_s T_{tr}\). The polyatomic model introduces two relaxation parameters, \(-0.5 \leq \nupESBGK < 1\), and \(0 \leq \theta \leq 1\), which relate to the relaxation temperature \(T_{rel}\) \begin{align}
	T_{rel}(t, x) = \theta T_{eq}(t, x) + (1 - \theta) T_{int}(t, x),
\end{align} and the covariance matrix \(\Ttensor\) \begin{align}
	\Ttensor = (1 - \theta) \left[ (1 - \nu)R_s T_{tr}(t, x) I_3 + \nu \Stensor \right] + \theta R_s T_{eq}(t, x) I_3. \label{eq:polyAtomicCovarianceMatrix}
\end{align} Using these definitions, we can define the polyatomic Maxwellian \(\mathcal{M}\) and Gaussian \(\mathcal{G}\) \begin{align}
	\mathcal{M}(t, x, v, I) &= \frac{\rho(t, x) \Lambda_\delta}{\left(2\pi R_s T_{eq}(t, x)\right)^{3/2}(R_s T_{eq})^{\delta/2}} \exp \left( -\frac{\abs{v - U(t,x)}^2}{2 R_s T_{eq}(t, x)} - \frac{I^{2/\delta}}{R_s T_{eq}(t, x)}  \right), \label{eq:Maxwellian} \\
	\mathcal{G}(t, x, v, I) &= \frac{\rho(t, x) \Lambda_\delta}{\sqrt{\det(2\pi \Ttensor)}(R_s T_{rel}(t, x))^{\delta/2}} \exp \left( -\frac{1}{2} \left(v - U(t, x)\right)^\top \Ttensor^{-1} \left(v - U(t, x)\right) - \frac{I^{2/\delta}}{R_s T_{rel}(t, x)} \label{eq:Gaussian} \right). 
\end{align} The constant \(\Lambda_\delta = \Gamma(\frac{\delta}{2} + 1)\) serves the purpose of normalising the term in the exponential related to the internal energy. Finally, we can define the ESBGK model for polyatomic gases \begin{align}
	\QpESBGK(f) = \frac{1}{\varepsilon} \left(\mathcal{G}(t, x, v, I) - f(t, x, v, I)\right),
\end{align} with \(\varepsilon = \mu \left(1 - \nupESBGK + \theta \nupESBGK\right)/p\). When the deriving the scheme, we allow the viscosity function to depend on the translational temperature \(\mu = \mu(T_{tr})\) as in \cite{andries2002}. Note that for \(\delta = 0\) and \(\theta = 0\), we have \(T_{eq} = T_{tr}\), and the polyatomic ESBGK model reduces to the monoatomic ESBGK model. 

The polyatomic ESBGK model conserves mass, momentum and energy, and admits to a H-theorem that was proven in \cite{andriesGaussianBGKModelBoltzmann2000}. As before, for vanishing Knudsen numbers, the model reduces to the Navier-Stokes-Fourier equations with correct Prandtl number. Moreover, the additional free parameter in the polyatomic model enables matching a further transport coefficient, known as the second viscosity.

\section{The semi-Lagrangian method and its properties}
\label{section:SL}
In this Section, we derive the semi-Lagrangian method for polyatomic ESBGK model. First, we introduce the notation. Then, in Section \ref{section:SLFirstOrder}, we derive the first-order method. In Section \ref{section:SLSecondOrder}, we extend the scheme to second-order in time and space using implicit linear multistep methods. In Section \ref{section:BCs}, we discuss how diffuse-reflective and pressure driven inflow and outflow boundary conditions are applied. Finally, in Section \ref{section:Chu}, we give a reduced order model of the polyatomic ESBGK equation for two-dimensional simulations in space.  

\subsection{Notation}
\label{section:SLNotation}
For simplicity, we consider a one-dimensional spatial domain \([x_L, x_R]\), with uniform grid spacing \(\Delta x = \frac{x_R - x_L}{N_x}\) such that \begin{align}
	x_i = x_L + \left(i + \frac{1}{2}\right) \Delta x, \quad i = 0 \ldots N_x - 1. 
\end{align} The velocity grid \(v \in \mathbb{R}^3\) is first truncated to \([-\vmax, \vmax]\) and subsequently discretised with mesh spacing \(\Delta v = \frac{2 \vmax}{N_v}\) in each direction: \begin{align}
	v_j = -\vmax \left(1, 1, 1\right)^\top + \left(j_1 \Delta v, j_2 \Delta v, j_3 \Delta v\right)^\top, \quad 0 \leq j_1, j_2, j_3 \leq N_v. 
\end{align} For more involved velocity discretisations, see \cite{brull2014, boscarinoLocalVelocityGrid2022}. Finally, we consider a fixed time step \(\Delta t\), such that \begin{align}
	t^n = n \Delta t, \quad n = 0 \ldots N_t, \quad N_t \Delta t = t_f.
\end{align} We keep the internal energy variable \(I \in \mathbb{R}^+\) continuous, but we omit it from the following sections, as it can be completely eliminated through a so-called Chu reduction \cite{chuKineticTheoreticDescriptionFormation1965, andriesGaussianBGKModelBoltzmann2000}. Further details are provided in Section \ref{section:Chu}. The numerical solution of \(f\) at \((t^n, x_i, v_j)\) is denoted by \(f_{i, j}^n\). Macroscopic quantities are approximated by discretising integrals over velocity space with the composite midpoint rule, e.g., \begin{align}
	\left(\rho_i^n, \rho_i^n U_i^n, E_i^n \right)^\top = \left< \sum_j \left((1, v_j, \frac{1}{2} \abs{v}^2) + I^{2/\delta}\right)^\top  f_{i, j}^n \Delta v^3 \right>_I, &\quad 
	\rho_i^n \Theta_i^n =  \sum_j (v_j - U_i^n) \otimes (v_j - U_i^n) \left< f_{i, j}^n \right>_I \Delta v^3 \\
	\Ttensor_i^n = (1 - \theta) \left[ (1 - \nu)R_s T_{tr, i}^n I_3 + \nu \Stensor_i^n \right] + \theta R_s T_{eq, i}^n I_3, &\quad \Sigma_i^n = \sum_j v_j \otimes v_j \left< f_{i, j}^n \right>_I \Delta v^3 \label{eq:covarianceMatrix} \\
	T_{tr, i}^n = \frac{2}{3 \rho_i^n R_s} \sum_j \frac{1}{2} \abs{v_j - U_i^n}^2 \left< f^n_{i, j} \right>_I \Delta v^3, &\quad T_{int, i}^n \frac{2}{\delta \rho_i^n R_s} \sum_j \left< I^{2/\delta} f^n_{i, j} \right>_I \Delta v^3 \\
	T_{eq, i}^n = \frac{1}{3 + \delta} \left(3 T_{tr, i}^n + \delta T_{int, i}^n \right), &\quad T_{rel, i}^n = \theta T_{eq}^n + (1 - \theta) T_{int, i}^n. \label{eq:temperatures2}
\end{align}
This procedure yields spectral accuracy in velocity, as long as the velocity \(\vmax\) is taken sufficiently large such that distribution \(f\) is small near the boundary of the velocity domain \cite{trefethen2014, groppi2014}. In the next section, we derive the first-order method. 

\subsection{The first-order method}
\label{section:SLFirstOrder}

\begin{figure}[htb]
	\centering
	\begin{tikzpicture}[scale=1.1, >=Stealth]
		
		% Axes
		\draw[->] (0,0) -- (6.5,0) node[right] {$x$};
		\draw[->] (0,-0.2) -- (0,2.5) node[above] {$t$};
		\draw[-] (0,1.5) -- (6,1.5);
		
		% Time levels
		\node[left] at (0,0) {$t^n$};
		\node[left] at (0,1.5) {$t^{n+1}$};
		
		% Spatial grid points
		\foreach \x/\lab in {1.2/$x_{i-2}$,2.4/$x_{i-1}$,3.6/$x_i$,4.8/$x_{i+1}$}
		{
			\fill (\x,0) circle (2pt);
			\fill (\x,1.5) circle (2pt);
			\node[below] at (\x,0) {\lab};
		}
		
		% Foot of characteristic
		\fill (1.7,0.0) circle (2pt);
		\node[below] at (1.7,0.0) {$\tilde{x}_i$};
		\node[above] at (1.7,0.0) {$\tilde f^n_{ij}$};
		
		% Point at next time
		\fill (3.6,1.5) circle (2pt);
		\node[above] at (3.6,1.5) {$f^{n+1}_{ij}$};
	
		% Characteristic line
		\draw[thick] (3.6,1.5) -- (1.7,0.0);
		\draw[dashed] (3.6,1.5) -- (3.6,0.4);
		\draw[dashed] (3.6,1.5) -- (4.8,0.4);
		\draw[dashed] (3.6,1.5) -- (2.9,0.4);
		\draw[dashed] (3.6,1.5) -- (4.3,0.4);
		
		% Velocity label
		\node[right] at (1.7,0.9) {$v_j>0$};
		
	\end{tikzpicture}
	\caption{The characteristic is traced back to the previous time step. The foot of the characteristic does not necessarily lie on the grid, thus interpolation is required \cite{groppi2014}.}
	\label{fig:SLCharacteristic}
\end{figure}
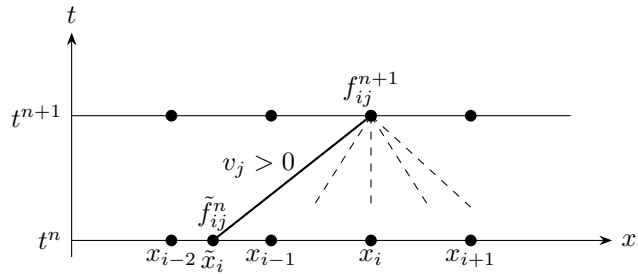
Using the method of characteristics, the kinetic equation \eqref{eq:BoltzmannEquation}, with polyatomic ESBGK collision operator \(\QpESBGK\), can be rewritten as an ordinary differential equation along characteristics \begin{align}
	\dv{f}{t} &= \frac{1}{\varepsilon} \left(\mathcal{G} - f\right), \quad f(0, x, v, I) = f^0(x, v, I), \label{eq:BGKODE} \\
	\dv{x}{t} &= v, \quad x(0) = \tilde{x}. \label{eq:characteristics}
\end{align} The characteristics, given by Equation \eqref{eq:characteristics}, define lines through space along which the distribution \(f\) only changes to collisions \eqref{eq:BGKODE}. To treat the potentially stiff collision operator, we use an implicit first-order method in time \begin{align}
	f^{n+1}_{i, j} = \tilde{f}^n_{i, j} + \frac{\Delta t}{\varepsilon^{n+1}_i}\left(\mathcal{G}^{n+1}_{i, j} - f^{n+1}_{i, j}\right), \label{eq:firstOrderImplicit}
\end{align} where \(\tilde{f}^n_{i, j}\) is the approximation of the solution at \(\tilde{x} = x_i - v_j \Delta t\) at time \(t^n\), see Figure \ref{fig:SLCharacteristic}. It is sufficient to use linear interpolation to obtain a first-order scheme. Equation \eqref{eq:firstOrderImplicit} constitutes a highly nonlinear system of equations in \(f_{i, j}^{n+1}\), due to the dependence of \(\mathcal{G}^{n+1}_{i, j}\) on \(\rho_i^{n+1}\), \(U_i^{n+1}\), \(\Ttensor_i^{n+1}\) and \(T_{rel, i}^{n+1}\). In addition, the relaxation time \(\varepsilon^{n+1}\) depends on the pressure \(p = \rho^{n+1}_i R_s T^{n+1}_{tr, i}\) and viscosity function \(\mu\left(T_{tr, i}^{n+1}\right)\). Nonetheless, using a trick analogous to \cite{boscarino2025a, filbet2011}, we show that implicit scheme can be evaluated cheaply without the need for a large nonlinear system solver. In the following three subsections, we progressively reduce the complexity of the nonlinear system. First, in Section~\ref{section:SL1}, we show that $\rho_i^{n+1}$ and $U_i^{n+1}$ can be computed explicitly. Next, in Section~\ref{section:SL2}, we derive a nonlinear system for $T_{tr,it}^{n+1}$ and $T_{int,it}^{n+1}$. Once these quantities are determined, the covariance matrix $\Ttensor_i^{n+1}$ can also be evaluated explicitly, 
as demonstrated in Section~\ref{section:SL3}.

\subsubsection{Explicit expressions for $\rho_i^{n+1}$ and $U_i^{n+1}$}
\label{section:SL1}
By multiplying Equation \eqref{eq:firstOrderImplicit} by the collision invariants \(\phi(v, I) = \left(1, v, \frac{1}{2} \abs{v}^2 + I^{2/\delta} \right)\), taking the summation over velocity index \(j\), and integrating over \(I\), we obtain \begin{align}
	\left< \sum_j \phi(v_j, I) \left(f^{n+1}_{i, j} - \tilde{f}^n_{i, j} \right) \Delta v^3 \right>_I = \frac{\Delta t}{\varepsilon^{n+1}_i} \left< \sum_j \phi(v_j, I) \left(\mathcal{G}^{n+1}_{i, j} - f^{n+1}_{i, j}\right) \Delta v^3 \right>_I. \label{eq:firstOrderImplicit1}
\end{align} Assuming the velocity domain \([-\vmax, \vmax]\) is chosen large enough, the right hand side in Equation \eqref{eq:firstOrderImplicit1} vanishes. Consequently, by conservation properties, we obtain, \begin{align}
	\rho_i^{n+1} = \tilde{\rho}_i^n, \quad U_i^{n+1} = \tilde{U}_i^n, \quad E_i^{n+1} = \tilde{E}_i^n, \label{eq:SLSchemeConservation} 
\end{align} where the tilde denotes macroscopic quantities derived from \(\tilde{f}\).

\subsubsection{Nonlinear system for the temperatures \(T_{tr,i}^{n+1}\) and \(T_{int,i}^{n+1}\)}
\label{section:SL2}
Next, we multiply Equation \eqref{eq:firstOrderImplicit} with \(\abs{v_j - U_i^{n+1}}^2\) and with \(I^{2/\delta}\). We integrate the resulting equations with respect to \(I\) and sum over the velocities to obtain, \begin{align}
	\frac{3}{2} \rho_i^{n+1} R_s \left[T_{tr, i}^{n+1} - \Tilde{T}_{tr, i}^{n+1}\right] &= \frac{\Delta t}{\varepsilon_i(T_{tr, i}^{n+1})} \left[\frac{3}{2}\rho_i^{n+1} R_s T_{eq, i}^{n+1} - \frac{3}{2}\rho_i^{n+1} R_s T_{tr, i}^{n+1} \right], \label{eq:SLTranslationalEnergy} \\
	 \frac{\delta}{2} \rho_i^{n+1} R_s \left[T_{int, i}^{n+1} - \Tilde{T}_{int, i}^{n+1}\right] &= \frac{\Delta t}{\varepsilon_i(T_{tr, i}^{n+1})} \left[\frac{\delta}{2} \rho_i^{n+1} R_s T_{rel, i}^{n+1} -\frac{\delta}{2} \rho_i^{n+1} R_s T_{int, i}^{n+1} \right] \label{eq:SLRotationalEnergy}
\end{align} Since the translational energy and rotational energy are not conserved, the right hand side in Equations \eqref{eq:SLTranslationalEnergy} and \eqref{eq:SLRotationalEnergy} do not vanish. Instead, after applying the definitions for \(T_{rel, i}^{n+1}\) and \(T_{eq, i}^{n+1}\), we obtain two coupled nonlinear equations for the translational and internal temperature, \begin{align}
	T_{tr, i}^{n+1} - \Tilde{T}_{tr, i}^n - \frac{\theta \Delta t}{\varepsilon^{n+1}_i(T_{tr, i}^{n+1})}\frac{\delta}{3 + \delta} \left[T_{int, i}^{n+1} - T_{tr, i}^{n+1}\right] = 0, \label{eq:SLTranslationalTemperature} \\
	T_{int, i}^{n+1} - \Tilde{T}_{int, i}^n - \frac{\theta \Delta t}{\varepsilon^{n+1}_i(T_{tr, i}^{n+1})}\frac{3}{3 + \delta} \left[T_{tr, i}^{n+1} - T_{int, i}^{n+1}\right] = 0. \label{eq:SLInternalTemperature}
\end{align} Equations \eqref{eq:SLTranslationalTemperature} and \eqref{eq:SLInternalTemperature} are solved using a Newton iteration. We note that this is a two-dimensional nonlinear system that must be solved at every spatial grid point, and is thus very cheap compared to the interpolations that are required to compute \(\Tilde{f}^n_{i}\). Given the temperatures \(T_{tr, i}^{n+1}\) and \(T_{int, i}^{n+1}\), the relaxation temperature \(T_{rel, i}^{n+1}\) and equilibrium temperature \(T_{eq, i}^{n+1}\) can be computed, see Equation \eqref{eq:temperatures2}.

\subsubsection{Explicit expression for the covariance matrix \(\mathcal{T}^{n+1}_i\)}
\label{section:SL3}
Next, we obtain an explicit expression for the covariance matrix \(\mathcal{T}^{n+1}_i\). First, we multiply Equation \eqref{eq:firstOrderImplicit} by \(v_i \otimes v_j\) and integrate \begin{align}
	\Sigma_i^{n+1} = \tilde{\Sigma_i^n} &+ \frac{\Delta t}{\varepsilon^{n+1}_i} \rho^{n+1}_i \left\{ \left(1-\theta\right) \left[(1-\nu)R_s T_{tr, i}^{n+1} I_3 + \nu \Stensor_i^{n+1}\right] + \theta R_s T_{eq, i}^{n+1} I_3 + U_i^{n+1} \otimes U_i^{n+1} \right\} -  \frac{\Delta t}{\varepsilon^{n+1}_i} \Sigma_i^{n+1} \\
	= \tilde{\Sigma_i^n} &+ \frac{R_s \Delta t}{\varepsilon^{n+1}_i} \rho_i^{n+1} \left\{(1-\theta)(1-\nu) R_s T_{tr, i}^{n+1} I_3 + \theta R_s T_{eq, i}^{n+1}I_3 + (1-(1-\theta)\nu) U_i^{n+1} \otimes U_i^{n+1} \right\} \label{eq:SigmaEquation} + \\ \quad \quad \quad &\Sigma_i^{n+1} \frac{\Delta t}{\varepsilon^{n+1}_i} ((1-\theta)\nu - 1), \nonumber
\end{align} where we have used the identity \begin{align}
	\rho_i^{n+1} \Stensor_i^{n+1} = \Sigma_i^{n+1} - \rho_i^{n+1} U_i^{n+1} \otimes U_i^{n+1}. \label{eq:StensorSigma}
\end{align} Equation \eqref{eq:SigmaEquation} is trivially solved for \(\Sigma_i^{n+1}\). Finally, using Equation \eqref{eq:StensorSigma} and the definition of the covariance matrix \(\Ttensor_i^{n+1}\), see Equation \eqref{eq:covarianceMatrix}, we obtain the following expression \begin{align}
	\mathcal{T}_i^{n+1} = \left[(1-\theta)(1-\nu)R_sT_{tr, i}^{n+1} + (1-\theta) \Delta t \frac{\omega_i^{n+1}}{\varepsilon} \left( (1-\theta)(1-\nu) R_s T_{tr, i}^{n+1} + \theta R_s T_{eq, i}^{n+1} \right) \right] I_3 \\ + (1-\theta)\omega_i^{n+1} \left[ \frac{\tilde{\Sigma}_i^{n}}{\rho_i^{n+1}} - U_i^{n+1} \otimes U_i^{n+1} \right], \nonumber 
\end{align} with \begin{align}
	\omega_i^{n+1} = \frac{\nu \varepsilon^{n+1}_i}{\varepsilon^{n+1}_i + \Delta t (1 - (1-\theta)\nu)}. 
\end{align} The Gaussian \(\mathcal{G}_{i}^{n+1}\) can now be evaluated fully explicitly using the macroscopic quantities derived from \(\tilde{f_i^n}\) and the expressions derived above. Then, Equation \eqref{eq:firstOrderImplicit} can be solved for \(f_{i, j}^{n+1}\) \begin{align}
	f_{i, j}^{n+1} = \frac{\tilde{f}_{i, j}^n \varepsilon^{n+1}_i + \Delta t G_{i, j}^{n+1} }{\varepsilon^{n+1}_i + \Delta t}. \label{eq:firstOrderImplicit2}
\end{align}

\subsection{Extension to higher orders}
\label{section:SLSecondOrder}
We extend the first-order scheme from Section \ref{section:SLFirstOrder} to second-order in time using a BDF2 method \begin{align}
	f^{n+1}_{i, j} - \frac{4}{3} f_{i, j}^{n, 1} + \frac{1}{3} f_{i, j}^{n, 2} = \frac{2}{3} \frac{1}{\varepsilon^{n+1}_i}\left(G_{i, j}^{n+1} - f^{n+1}_{i, j}\right), \label{eq:SLHigherOrder}
\end{align} where \(f_{i, j}^{n, k}\) is an approximation for \(f(t^n, x_i - k v_j \Delta t, v_j, I))\). The quantities \(f_{i, j}^{n, k}\) are interpolated from the grid at the previous time steps using a third-order bicubic reconstruction routine. For more involved reconstruction procedures, see \cite{groppi2014, choConservativeSemiLagrangianSchemes2021a, cho2021}. The scheme \eqref{eq:SLHigherOrder} can be written in the following way \begin{align}
	f^{n+1}_{i, j} = \tilde{f}^{n}_{i, j} + \frac{2}{3} \frac{1}{\varepsilon^{n+1}_i}\left(G_{i, j}^{n+1} - f^{n+1}_{i, j}\right), \text{with } \tilde{f}^{n}_{i, j} = \frac{4}{3} f_{i, j}^{n, 1} - \frac{1}{3} f_{i, j}^{n, 2}. 
\end{align} Using the procedure outlined in Section \ref{section:SLFirstOrder}, the quantity \(G_{i, j}^{n+1}\) can again be computed explicitly, this time from the linear combination of moments of \(f_{i, j}^{n, 1}\) and \(f_{i, j}^{n, 2}\) through \(\tilde{f}^{n}_{i, j}\). Furthermore, because BDF methods are asymptotic‑preserving (AP) and stiffly accurate (SA), they recover, in the limit \(\varepsilon \to 0\), a scheme of the same order for the Euler equations.
\begin{remark}
	We note that the central bicubic reconstruction proposed here will yield oscillations at shocks in the \(\varepsilon \to 0\) limit. Since our investigation is restricted to the slip and transition regime, where such features do not develop, central higher-order reconstructions are safe and efficient.
\end{remark}

\subsection{Boundary conditions}
\label{section:BCs}
\begin{figure}[htb]
	\centering
	\begin{tikzpicture}[scale=1.1, >=Stealth]
		
		% Axes
		\draw[->] (0,0) -- (6.5,0) node[right] {$x$};
		\draw[->] (0,0) -- (0,2.5) node[above] {$t$};
		\draw[-] (0,1.5) -- (6,1.5);
		
		% Time levels
		\node[left] at (0,0) {$t^n$};
		\node[left] at (0,1.5) {$t^{n+1}$};
		
		% Spatial grid points
		\fill (0,0) circle (2pt);
		\fill (0,1.5) circle (2pt);
		\node[below] at (0,0) {$x_0$};
		\foreach \x/\lab in {1.2/$x_{1}$,2.4/$x_{2}$,3.6/$x_3$}
		{
			\fill (\x,0) circle (2pt);
			\fill (\x,1.5) circle (2pt);
			\node[below] at (\x,0) {\lab};
		}		
		% Foot of characteristic
		%\fill (1.7,0.0) circle (2pt);
		%\node[below] at (1.7,0.0) {$\tilde{x}_i$};
		%\node[above] at (1.7,0.0) {$\tilde f^n_{ij}$};
		
		% Point at next time
		%\fill (3.6,1.5) circle (2pt);
		%\node[above] at (3.6,1.5) {$f^{n+1}_{ij}$};
		
		% Characteristic line
		\draw[dashed] (0,1.5) -- (-1.8,0.0);
		\draw[dashed] (0,1.5) -- (-0.9,0.0);
		\draw[thick] (0,1.5) -- (0.9,0.0);
		\draw[thick] (0,1.5) -- (1.8,0.0);
		\draw[thick] (0,1.5) -- (2.7,0.0);
		
		% Velocity label
		% \node[right] at (1.7,0.9) {$v_j>0$};
		
	\end{tikzpicture}
	\caption{At a boundary, characteristics for which \(v_j \cdot n \geq 0\) require a boundary condition.}
	\label{fig:SLCharacteristicBC}
\end{figure}
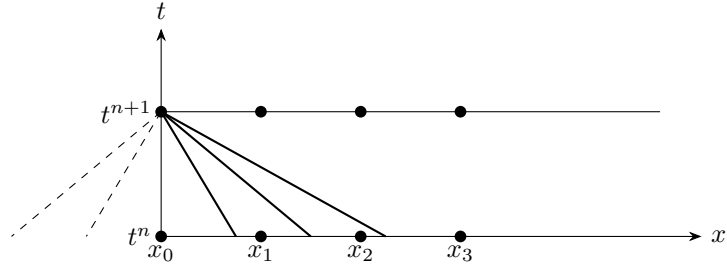
At a point boundary point, the characteristic feet for which \(v_j \cdot \normal \leq 0\), with \(\normal\) the inward-pointing normal, lie inside the gas domain. The solution at these points can be reconstructed in the usual way. The characteristics satisfying \(v_j \cdot \normal \geq 0\) fall outside the gas domain, see Figure \ref{fig:SLCharacteristicBC}, and hence require a boundary condition. After applying the boundary condition, the full distribution \(\Tilde{f}^n_i\) is known and the collision step can be performed, e.g. for the first order scheme see Equation \eqref{eq:firstOrderImplicit2}.

In this text, we consider three types of boundary conditions: specular reflection, diffuse-reflection and subsonic inflow/outflow. In the case of specular reflection, the distribution function \(\Tilde{f}^n_{i, j}\) is mapped onto \(\Tilde{f}^n_{i, k}\), where \(v_k\) is the velocity with the same tangential velocity and opposite normal velocity \begin{align}
	v_k = v_j - 2 \normal \left(v_j \cdot \normal \right). 
\end{align} We note that at a specular boundary, the internal energy remains unchanged. In the case of diffuse-reflective boundary conditions, the distribution of particles striking the wall is reflected back inside the domain as a Maxwellian with wall velocity \(U_w\) and wall temperature \(T_w\). The boundary condition reads \begin{align}
	\Tilde{f}^n_{i,j} = \mathcal{M}^n_{w, i, j} =  \frac{\rho_w \Lambda_\delta}{\left(2\pi R_s T_w\right)^{3/2}(R_s T_w)^{\delta/2}} \exp \left( -\frac{\abs{v_j - U_w}^2}{2 R_s T_w} - \frac{I^{2/\delta}}{R_s T_w}  \right), \quad v_j \cdot \normal \geq 0. \label{eq:DRBC}
\end{align} The density of the gas that is reflected is such that the flux at the boundary is zero, or in other words, mass is conserved across the boundary, \begin{align}
	\left< \sum_{j, (v_j - U_w) \cdot \normal > 0} \mathcal{M}^n_{w, i, j}\,(v_j - U_w) \cdot \normal \right>_I + \left< \sum_{j, (v_j - U_w) \cdot \normal  \leq 0} \Tilde{f}^n_{i,j} \, (v_j - U_w) \cdot \normal \right>_I = 0. \label{eq:boundaryFlux}
\end{align} Equation \eqref{eq:boundaryFlux} can be solved for the density \(\rho_w\) after reconstructing the distribution at the characteristic feet for all velocities that satisfy \(v_j \cdot \normal \leq 0\). Then, the distribution for the outgoing characteristics is obtained from Equation \eqref{eq:DRBC}. 

Finally, we consider the subsonic inflow and outflow boundary condition. We assume the distribution of the gas flowing in/out of the domain is in thermodynamic equilibrium. In other words, we assume the distribution \(\Tilde{f}^n_{i,j}, v_j \cdot \normal \geq 0\) is a Maxwellian with some density \(\rho_f\), velocity \(U_f\) and temperature \(T_f\). The quantities \(\rho_f\), \(U_f\) and \(T_f\) are then determined from boundary conditions for the Euler equations. First, the density, mean velocity, and temperature are extrapolated from inside the domain to the boundary. We denote these values with subscript \(e\). Then, we define the extrapolated speed of sound and extrapolated pressure \begin{align}
	p_{e} = \rho_{e} R_s T_{e} \\
	c_{e} = \sqrt{\frac{\delta + 5}{\delta + 3} T_{e} R_s}. 
\end{align} Then, for the subsonic inflow boundary condition, given the inflow pressure \(p_{in}\) and inflow temperature \(T_{in}\), the inflow quantities are defined as \begin{align}
	\rho_f = \frac{p_{in}}{T_{in} R_s}, \quad U_f = U_{e} + \frac{p_{in} - p_{e}}{\rho_{e} c_{e}} \normal, \quad T_f = \frac{p_{in}}{R_s \rho_f}. 
\end{align} In the case of subsonic outflow boundary conditions, we prescribe the outflow pressure \(p_o\), and set the outflow quantities to \begin{align}
	\rho_f = \rho_{e} + \frac{p_o - p_{e}}{c_{e}c_{e}}, \quad U_f = U_{e} + \frac{p_{o} - p_{e}}{\rho_{e} c_{e}} \normal, \quad T_f = \frac{p_{in}}{R_s \rho_f}. 
\end{align} For a derivation of boundary conditions for the Euler equations, see \cite{whitfield1984}. We note that inflow/outflow boundary conditions based on the Euler equations are also used in DSMC simulations \cite{wang2004, nance1997}. 
\begin{remark}
	It may happen that characteristics from points that do not lie on the boundary fall outside the gas domain. For simplicity, we limit the time step so that this does not occur. However, through the use of ghost points and additional interpolation, the procedure above can be generalised \cite{groppiBoundaryConditionsSemiLagrangian2016}.  
\end{remark}

\begin{remark} % Reconstruction
	At a boundary, in the case of the higher-order method, it is not possible to use a central third-order reconstruction. Instead, we resort to one-sided stencils. The use of different stencils breaks the translation invariance property of the scheme and may introduce conservation errors, especially in the \(\varepsilon \to 0\) limit \cite{choConservativeSemiLagrangianSchemes2021a}. Since our investigation is restricted to slip and transition flows, this is not an issue.  
\end{remark}

\subsection{Chu-reduction} % Chu-reduction
\label{section:Chu}
In the numerical examples in Section \ref{section:Numerics} we will restrict ourselves two-dimensional simulations in space. Consider the velocity component orthogonal to the spatial domain \(v_3\), with \(v = \left(v_1, v_2, v_3\right)^\top\), and consider the internal energy variable \(I\). Although the simulations should include effects due to motion in three dimensions and internal energy modes, the full distribution in \(v_3\) or \(I\) is irrelevant. Ultimately, some specific moments, such as the mean velocity \(U\) or the internal temperature \(T_{int}\) are the quantities of interest. In addition, it is expensive to discretise the dimensions \(v_3\) and \(I\). It turns out that for BGK-type models it is not necessary to consider the distribution \(f(t, x, v, I)\)'s dependence on \(v_3\) and \(I\). Instead, the same result can be achieved using reduced distributions. This procedure is known as a Chu-reduction \cite{chuKineticTheoreticDescriptionFormation1965, andriesGaussianBGKModelBoltzmann2000}. We define three reduced distribution functions \begin{align}
	f_1(t, x, \vbar) &= \int_{\mathbb{R}\cross \mathbb{R}^+} f(t, x, v, I) dv_3 dI, \quad
	f_2(t, x, \vbar) = \int_{\mathbb{R}\cross \mathbb{R}^+} v_3^2 f(t, x, v, I) dv_3 dI, \\	
	f_3(t, x, \vbar) &= \int_{\mathbb{R}\cross \mathbb{R}^+} I^{\delta/2} f(t, x, v, I) dv_3 dI,
\end{align} with \(\vbar = \left(v_1, v_2\right)^\top\). We assume the mean velocity in the direction orthogonal to the two-dimensional space \(x\) is zero. In addition, we assume the distribution function \(f(t, x, v, I)\) is symmetric in \(v_3\). The conserved quantities defined in Section \ref{section:polyatomicESBGK} now become \begin{align}
	\rho(t, x) = \left< f_1 \right>, \quad \rho(t, x) \underline{U}(t, x) = \left< f_1 \right>, \quad E(t, x) = \left< \frac{1}{2} \abs{\vbar}^2 f_1 \right> + \left< \frac{1}{2} f_2 \right> + \left< f_3 \right>. 
\end{align} The two internal energy components become \begin{align}
	e_{tr}(t, x)  &= \frac{1}{\rho(t, x)} \left<
	\tfrac{1}{2} \abs{\vbar - \underline{U}(t,x)}^2 f_1
	\right> + \left< \frac{1}{2} f_2 \right> = \frac{3}{2}\, R_s\, T_{tr}(t, x), \\
	e_{int}(t, x) &= \frac{1}{\rho(t, x)} \left< f_3 \right> = \frac{\delta}{2}\, R_s\, T_{int}(t, x). 
\end{align} Finally, we define the reduced stress tensor \(\underline{\Stensor} \in \mathbb{R}^{2 \cross 2}\) and reduced covariance matrix \(\underline{\Ttensor} \in \mathbb{R}^{2 \cross 2}\) as \begin{align}
	\rho(t, x) \Stensor = \begin{pmatrix}
		\rho(t, x) \underline{\Stensor} & 0 \\
		0 & \left< f_2 \right>
	\end{pmatrix} \text{ with }\, \rho(t, x) \underline{\Stensor} = \left< (\vbar - \underline{U}(t, x)) \otimes (\vbar - \underline{U}(t, x)) f_1 \right> \\
	\Ttensor = (1 - \theta) \left[ (1 - \nu)R_s T_{tr}(t, x) I_3 + \nu \Stensor \right] + \theta R_s T_{eq}(t, x) I_3 = \begin{pmatrix}
		\underline{\Ttensor} & 0 \\
		0 & \Ttensor_{33}
	\end{pmatrix}. 
\end{align} The reduced distribution functions \(f_1\) and \(f_2\) contain all mass and momentum, and the distribution function \(f_3\) only contributes to the internal energy. By multiplying the polyatomic ESBGK equation with \(1\), \(v_3^2\) and \(I^{2/\delta}\), and integrating over \(v_3\) and \(I\), we obtain three BGK-type equations for the reduced distribution functions \begin{align}
	\pdv{f_1}{t} + v \cdot \nabla_x f_1 = \frac{1}{\varepsilon}(G_1 - f_1), \quad
	\pdv{f_2}{t} + v \cdot \nabla_x f_2 = \frac{1}{\varepsilon}(\Ttensor_{33} G_1 - f_2), \quad		
	\pdv{f_3}{t} + v \cdot \nabla_x f_3 = \frac{1}{\varepsilon}(\frac{\delta}{2} R T_{rel}G_1 - f_3). \label{eq:ChuEquations}
\end{align} The attractor in Equation \eqref{eq:ChuEquations} is defined by \begin{align}
	G_1 = \frac{\rho(t, x) \sqrt{\Ttensor_{33}}}{2\pi \sqrt{\det(\Ttensor_i)}} \exp \left( - \frac{1}{2} \left( \vbar -\underline{U}(t, x) \right)^\top \underline{\Ttensor}^{-1} \left( \vbar -\underline{U}(t, x) \right) \right). 
\end{align} Thus, we have reduced the dimension of the equation from eight to five, thereby significantly speeding up simulations.

\section{Asymptotic preserving property and Navier-Stokes limit}
\label{section:AP}
In this section, we study several properties of the first-order scheme \eqref{eq:firstOrderImplicit}. The properties below are generalisations of the ESBGK model \cite{filbet2011, boscarino2025a}. 
\begin{proposition}
	For \(\tau > 0\) and \(\Delta t > 0\), the solution of \eqref{eq:firstOrderImplicit} satisfies \begin{align}
		0 \leq f_{i, j}^{n+1} \leq \max_{i, j} \left\{\norm{f^n_{i, j}}, \norm{\mathcal{G}^{n+1}_{i, j}}\right\}. 
	\end{align}
\end{proposition}
\textbf{\textit{Proof}} The solution \(f_{i, j}^{n+1}\) is a convex combination of \(\mathcal{G}^n_{i, j}\) and \(\tilde{f}_{i, j}^n\). The value \(\tilde{f}_{i, j}^n\) is the result of linear interpolation of values \(f^n_{:, j}\). This implies the desired result. 

The following proposition concerns the asymptotic‑preserving (AP) property of the scheme. A numerical scheme is said to be AP if, in the limit as the Knudsen number tends to zero, it naturally reduces to a consistent discretisation of the Euler equations. Before this limit is reached, the underlying kinetic equation converges asymptotically to the Navier-Stokes equations with well‑defined transport coefficients. The scheme under consideration can be shown to possess this property as well.
\begin{proposition}
	Let \(f^n = f(t^n, x, v, I)\) be a time-discretised solution of the polyatomic ESBGK equation. Then, assuming that the following quantities are bounded by positive constants \(C_1, C_2, C_3\) and \(C_4\) independent of \(n\), \begin{align}
		\sum_{0 \leq \alpha \leq 2} \sup_{x, v, I} \left| \partial_x^\alpha f^n(x, v, I) \left(1 +|v|^2\right)^4 \right| < C_1, \quad \sum_{0 \leq \alpha \leq 2} \sup_{x, v, I} \left| \partial_x^2 f^n(x, v, I) |v|^2 I^{2/\delta} \right| < C_2 \\ 
		\sum_{0 \leq \alpha \leq 2} \sup_{x, v, I} \left| \partial_x^\alpha \mathcal{M}^\alpha(x, v, I) \left(1 +|v|^2\right)^4 \right| < C_3, \quad \sum_{0 \leq \alpha \leq 2} \sup_{x, v, I} \left| \partial_x^2 \mathcal{M}^\alpha(x, v, I) |v|^2 I^{2/\delta} \right| < C_4,
	\end{align} the first-order semi-Lagrangian scheme \eqref{eq:firstOrderImplicit} asymptotically becomes a consistent approximation of the Navier-Stokes system \eqref{eq:NS}. 
\end{proposition}
\textbf{\textit{Proof}} We consider the first order scheme \eqref{eq:firstOrderImplicit}, and leave the space, velocity and internal energy variable continuous. We then take the moments with respect to the polyatomic collision invariants \(\phi(v, I)\) \begin{align}
	\int_{\mathbb{R}^3 \cross \mathbb{R}^+} \phi(v, I) f^{n+1} dvdI &= \int_{\mathbb{R}^3 \cross \mathbb{R}^+} \phi(v, I) \tilde{f}^n dvdI, \label{eq:NSProof1}
\end{align} with \(\tilde{f}^n = f(x - v\Delta t, v, I)\). Using a Taylor expansion, we expand the right hand side of Equation \eqref{eq:NSProof1}, \begin{align}
	\int_{\mathbb{R}^3 \cross \mathbb{R}^+} \phi(v, I) f^{n+1} dvdI &= \int_{\mathbb{R}^3 \cross \mathbb{R}^+} \phi(v, I) \left[ f^n(x, v, I) - \Delta t v \cdot \nabla_x f^n(x, v, I) \right] dvdI + R, \label{eq:NSProof2}
\end{align} where \(R\) is the remainder term in integral form. Using similar techniques from \cite{boscarino2025a}, it is possible to bound the remainder term \begin{align}
	\norm{R}_{\infty} &\leq \frac{3 \pi \Delta t^2}{8} \sup_{x, v, I} \left| \partial_x^2 f^n(x - v \Delta t, v, I) \left(1 +|v|^2\right)^4 \right| + \frac{\Delta t^2}{2} \sup_{x, v, I} \left| \partial_x^2 f^n(x, v, I) |v|^2 I^{2/\delta} \right| \\
	&= \mathcal{O}\left(\Delta t^2\right). 
\end{align} We now treat the relaxation time $\varepsilon$ as independent of $f^{n}(x, v, I)$ and formally expand both the distribution function and the Gaussian in powers of \(\varepsilon\), \begin{align}
	f^n = \mathcal{M}^n + \varepsilon f^n_1 + \mathcal{O}(\varepsilon^2), \quad
	\mathcal{G}^n = \mathcal{M}^n + \varepsilon \mathcal{G}^n_1 + \mathcal{O}(\varepsilon^2). \label{eq:Expansions}
\end{align} By plugging these expansions in the definition of a collision invariant \eqref{eq:collisionInvariant}, we obtain the so-called compatibility conditions \begin{align}
	\left< \left< f_1^n \left(1, v, \frac{1}{2} \abs{v}^2 + I^{2/\delta} \right)^\top \right> \right> = \left< \left< \mathcal{G}_1^n \left(1, v, \frac{1}{2} \abs{v}^2 + I^{2/\delta} \right)^\top \right> \right> = 0. \label{eq:CollisionInvariantCheck}
\end{align} Additionally, we expand the translational temperature, the internal temperature and the stress tensor like such \begin{align}
	T_{tr}^n &= T_{eq}^n + \varepsilon T_{tr}^{\varepsilon, n} + \mathcal{O}(\varepsilon^2), \\
	T_{int}^n &= T_{eq}^n + \varepsilon T_{int}^{\varepsilon, n} + \mathcal{O}(\varepsilon^2), \\
	\Stensor^n &= R_s \left(T_{eq}^n + \varepsilon T_{tr}^{\varepsilon, n}\right) I_3 + \varepsilon \Stensor^n_1 + \mathcal{O}(\varepsilon^2). 
\end{align} A necessary condition to satisfy the compatibility conditions \eqref{eq:CollisionInvariantCheck} is that \(\tr \Stensor^n_1 = 0\). Due to the definition of the internal energies and the stress tensor, we find that \begin{align}
	\frac{3}{2} \rho^n R_s T_{tr}^{\varepsilon, n} &= \left< \left< \frac{|v - U|^2}{2} f_1 \right> \right> \label{eq:transportTemperature}, \\
	\frac{\delta}{2} \rho^n R_s T_{int}^{\varepsilon, n} &= \left< \left<I^{2/\delta} f_1 \right> \right> \label{eq:internalTemperature}, \\
	\rho^n R_s T_{tr}^{\varepsilon, n} + \rho^n \Stensor^n_1 &= \left< \left<(v - U^n) \otimes (v - U^n) f_1 \right> \right>. \label{eq:stressTensor}
\end{align} By substituting a zero-th order expansion \(f^n = \mathcal{M}^n\) into Equation \eqref{eq:NSProof2}, we obtain a consistent scheme for the Euler equations
\begin{equation}
	\begin{aligned}
		\frac{\rho^{n+1} - \rho^n}{\Delta t} + \nabla_x \cdot (\rho^n U^n) &= \mathcal{O}(\Delta t) + \mathcal{O}(\varepsilon), \\
		\frac{\rho^{n+1} U^{n+1} - \rho^{n} U^{n}}{\Delta t} + \Delta t \nabla_x \cdot \left(\rho^n U^n \otimes U^n + p I_3 \right) &= \mathcal{O}(\Delta t) + \mathcal{O}(\varepsilon), \\
		\frac{E^{n+1} - E^n}{\Delta t} + \nabla_x \cdot \left( U^n E^n \right) &= \mathcal{O}(\Delta t) + \mathcal{O}(\varepsilon).
	\end{aligned} \label{eq:SLEuler}
\end{equation} Inserting the first‑order expansion of the distribution function \(f^n = \mathcal{M}^n + \varepsilon f^n_1\) into Equation \eqref{eq:NSProof2} yields a consistent discretisation of the Navier–Stokes equations
\begin{equation}
	\begin{aligned}
		\frac{\rho^{n+1} - \rho^n}{\Delta t} + \nabla_x \cdot (\rho^n U^n) &= \mathcal{O}(\Delta t) + \mathcal{O}(\varepsilon^2), \\
		\frac{\rho^{n+1} U^{n+1} - \rho^{n} U^{n}}{\Delta t} + \Delta t \nabla_x \cdot \left(\rho^n U^n \otimes U^n + p I_3 \right) &= - \varepsilon \nabla_x \cdot \left(\rho^n \Theta_1^n \right) + \mathcal{O}(\Delta t) + \mathcal{O}(\varepsilon^2), \\
		\frac{E^{n+1} - E^n}{\Delta t} + \nabla_x \cdot \left( U^n E^n \right) &= - \varepsilon \nabla_x \left( q_1^n - \rho^n \Theta_1^n U^n \right) + \mathcal{O}(\Delta t) + \mathcal{O}(\varepsilon^2).
	\end{aligned} \label{eq:SLNavierStokes}
\end{equation} where \begin{align}
	q_1^n = \int_{\mathbb{R}^3 \cross \mathbb{R}^+} \frac{\abs{v - U^n}^2}{2} (v - U^n) f_1^n dv dI,
\end{align} is the heat flux, and \(\rho \Stensor_1^n\) the viscosity tensor. In the text below, we apply a Chapman-Enskog expansion to obtain expressions for \(f_1^n\) and \(\mathcal{G}^n_1\). We then use these expressions to explicitly compute the transport coefficients in the Navier-Stokes equation \eqref{eq:SLNavierStokes}. The first step is to fill in the expansions \eqref{eq:Expansions} into \eqref{eq:firstOrderImplicit}, which yields \begin{align}
	\frac{\mathcal{M}^{n+1} - \Tilde{\mathcal{M}}^n}{\Delta t} = \mathcal{G}_1^{n+1} - f_1^{n+1} - \varepsilon \frac{f_1^{n+1} - \Tilde{f}_1^n}{\Delta t} + \mathcal{O}\left(\varepsilon^2\right). \label{eq:f1A}
\end{align} We Taylor-expand the terms with a tilde and solve the resulting equation for \(f_1^n\). For brevity, we drop all higher-order terms in \(\Delta t\) and \(\varepsilon\), and obtain \begin{align}
	f^{n+1}_1 = \mathcal{G}_1^{n+1} - \frac{\mathcal{M}^{n+1} - \mathcal{M}^n}{\Delta t} - v \cdot \nabla_x \mathcal{M}^n. \label{eq:f1B}
\end{align} Below, all three terms in the right hand side of Equation \eqref{eq:f1B} are approximated. It is the zero-th order term in \(\varepsilon\) that we require. 

Next, we expand \(\mathcal{G}_1^{n+1}\) in terms of \(\varepsilon\). We note that the covariance matrix \(\mathcal{T}^n\) and relative temperature \(T_{rel}^n\) both have contributions on the order \(\varepsilon\) due to the definitions of the translational temperature \eqref{eq:transportTemperature}, internal temperature \eqref{eq:internalTemperature} and stress tensor \eqref{eq:stressTensor}. Consequently, the prefactor and the exponent of \(\mathcal{G}_1^{n+1}\) \eqref{eq:Gaussian} must be expanded in \(\varepsilon\). We start with the covariance matrix \(\mathcal{T}_n\) \eqref{eq:polyAtomicCovarianceMatrix}  \begin{align}
	\mathcal{T}^n &= I_3 [R_s T_{eq}^n + (1-\theta) \varepsilon R_s T_{tr}^{\varepsilon, n} ] + (1- \theta) \nu \varepsilon \Theta_1^n + \mathcal{O}(\varepsilon^2). 
\end{align} The determinant and inverse of \(\mathcal{T}^n\) can be approximated as follows \begin{align}
	\det(\mathcal{T}^n) = (R_s T_{eq}^n)^3 + 3 \varepsilon (1 - \theta) R_s T_{tr}^{\varepsilon, n} (R_s T_{eq}^n)^2 + \mathcal{O}(\varepsilon^2), \label{eq:covarianceDet} \\
	\left( \mathcal{T}^n \right)^{-1} = \frac{1}{R_s T_{eq}^n} I_3 - \frac{\varepsilon (1 - \theta) }{(R_s T_{eq}^n)^2} \left[ I_3 R_s T_{tr}^{\varepsilon, n} + \nu \Theta_1^n \right] + \mathcal{O}(\varepsilon^2), \label{eq:covarianceInverse}
\end{align} where in the last line, we have used the Woodbury matrix identity. Using Equation \eqref{eq:internalTemperature}, we can write the relaxation temperature as \begin{align}
	T_{rel}^n = \theta T_{eq}^n + (1 - \theta) T_{int}^n = T_{eq}^n + \varepsilon (1 - \theta) T_{int}^{\varepsilon, n} + \mathcal{O}(\varepsilon^2). \label{eq:relaxationTemperature}
\end{align} Then, by filling \eqref{eq:relaxationTemperature}, \eqref{eq:covarianceInverse} and \eqref{eq:covarianceDet} into \eqref{eq:Gaussian}, and applying further Taylor expansions, we obtain an approximation for \begin{align}
	\mathcal{G}_1^n &= \frac{\mathcal{G}^n - \mathcal{M}^n}{\varepsilon} = M[f^n] \frac{(1 - \theta)}{2} \left( -\frac{3 T_{tr}^\varepsilon}{T_{eq}} + \frac{1}{(R T_{eq})^2} (v - U)^T \left[R T_{tr}^\varepsilon I_3  + \nu \Theta_1 \right] + \frac{2 R T_{int}^\varepsilon I^{2/\delta}}{(R T_{eq})^2} - \frac{\delta T_{int}^\varepsilon}{T_{eq}} \right). \label{eq:term1}
\end{align} It remains to obtain an expression for the final two terms in \eqref{eq:f1B}. The differential of the Maxwellian \(\mathcal{M}^n\) is given by \begin{align}
	d\mathcal{M}^n = \mathcal{M}^n d \ln \mathcal{M}^n = \mathcal{M}^n \left[ \frac{d\rho}{\rho} + \frac{dT_{eq}}{T_{eq}^n} \left( -\frac{3 + \delta}{2} + \frac{|v - U|^2}{2 R T_{eq}} + \frac{I^{2/\delta}}{R T_{eq}} \right) + \frac{(v - U)}{R T_{eq}} dU\right]
\end{align} Then, we can approximate the final two terms in \eqref{eq:f1B} as \begin{align}
	\frac{\mathcal{M}^{n+1} - \mathcal{M}^n}{\Delta t} = \mathcal{M}^n \left[\frac{\rho^{n+1} - \rho^n}{\Delta t} \frac{1}{\rho^n}  + \frac{T_{eq}^n - T_{eq}^n}{\Delta t}\frac{1}{T_{eq}^n}\left( -\frac{3+\delta}{2} + \frac{I^{2/\delta}}{R T_{eq}^n} + \frac{|v-U^n|^2}{2 R T_{eq}^n} \right) + \frac{(v - U^n)}{R T_{eq}^n} \cdot \frac{U^{n+1} - U^n}{\Delta t} \right] \label{eq:term2}
\end{align} and \begin{align}
	v \cdot \mathcal{M}^n = \mathcal{M}^n \left[ \frac{v \cdot \nabla_x \rho^n}{\rho^n} + \frac{v \cdot \nabla_x T_{eq}^n}{T_{eq}^n}\left( -\frac{3+\delta}{2} + \frac{I^{2/\delta}}{R T_{eq}^n} + \frac{|v-U^n|^2}{2 R T_{eq}^n} \right) + \frac{(v - U^n)}{R T_{eq}^n} v \cdot \nabla_x U^n \right].  \label{eq:term3}
\end{align} By gathering \eqref{eq:term1}, \eqref{eq:term2} and \eqref{eq:term3}, the first-order contribution in \(\varepsilon\) to the distribution function \(f^n\) is \begin{align} \begin{split}
		f_1^n &= \frac{(1-\theta)}{2} \left[ -\frac{3 T_{tr}^{\varepsilon, n}}{T_{eq}^n} + \frac{1}{(R_s T_{eq}^n)^2} (v - U^n)^T \left[R_s T_{tr}^{\varepsilon, n} I_3 + \nu \Stensor_1^n \right] (v - U^n) + \frac{2 R_s T_{int}^{\varepsilon, n}}{(R_s T_{eq}^n)^2} - \frac{\delta T_{int}^{\varepsilon, n}}{T_{eq}^n} \right] \\
		-& \frac{\mathcal{M}^n}{R_s T_{eq}^n } (v - U^n) \nabla_x U^n (v - U^n) - \mathcal{M}^n\left( -\frac{3+\delta + 2}{2} + \frac{I^{2/\delta}}{R_s T_{eq}^n} + \frac{|v-U^n|^2}{2 R_s T_{eq}^n} \right) \left( \frac{1}{ T_{eq}^n} (v - U^n) \cdot \nabla_x T_{eq}^n \right) \\
		+& \left( \frac{I^{2/\delta}}{R_s T_{eq}^n} + \frac{|v-U^n|^2}{2 R_s T_{eq}^n} \right) \frac{2}{3 + \delta} \nabla_x \cdot U^n. \label{eq:f1n}
	\end{split} 
\end{align} The expression for the perturbation \(f_1^n\) still depends on the translational and internal perturbation temperatures \(T_{tr}^{\varepsilon, n}\) and \(T_{int}^{\varepsilon, n}\). These can now both be computed from their definitions \eqref{eq:transportTemperature} and \eqref{eq:internalTemperature}. Since the computations involve only standard algebraic manipulations and higher‑order moments of Maxwellians, we omit the details. We find for the perturbation temperatures \begin{align}
	T_{tr}^{\varepsilon, n} = \frac{-2\delta}{(3 + \delta) 3} \frac{1}{\theta} T_{eq}^n \nabla_x \cdot U^n, \quad
	T_{int}^{\varepsilon, n} = \frac{2}{3 + \delta} \frac{1}{\theta} T_{eq}^n \nabla_x \cdot U^n. 
\end{align} It is straightforward to check that \(f_1^n\) satisfies the compatibility conditions \eqref{eq:CollisionInvariantCheck} with these expressions for the perturbation temperatures. It remains to explicitly compute the stress tensor and heat flux, \begin{align}
	\rho^n \Stensor^n &= \rho^n R_s T_{eq}^n I_3 - \mu \left[ \nabla_x U^n +  \left(\nabla_x U^n\right)^\top - \frac{2I_3}{3 + \delta} \nabla_x \cdot U^n - \frac{(1-\theta)}{\theta} (1 - \nu) \frac{2 \delta I_3 }{3(3 + \delta)} \nabla_x \cdot U^n \right], \\
	q_1^n &= - \frac{5 + \delta}{2} \mu (1 - \nu + \theta\nu) R_s \nabla_x T_{eq}^n. 
\end{align} By comparing the stress tensor and heat flux obtained via a Chapman–Enskog expansion of the continuous-in-time polyatomic ESBGK model \cite{andriesGaussianBGKModelBoltzmann2000}, we find identical transport coefficients. This demonstrates that the first-order scheme asymptotically converges to a discretization of the Navier–Stokes equations, and concludes the proof. 

\section{Numerical simulations}
\label{section:Numerics}
In this section, we present several numerical examples to illustrate the performance of our numerical method in comparison with existing tools, and to examine how the polyatomic ESBGK model contrasts with simpler BGK formulations, DSMC simulations, and the Navier–Stokes equations. We have implemented the first-order and second-order scheme for the Chu-reduced formulation, see Section \ref{section:Chu}, for the BGK, ESBGK and polyatomic ESBGK models. All simulations are performed in 2D on a Cartesian mesh with mesh spacing \(\Delta x\) and \(\Delta y\). Time steps are quantified using a CFL number defined as \(\text{CFL} = \min\left(\Delta x, \Delta y\right)/\vmax\). The code is available at \cite{willems_sl4bgk_2026}. 

All simulations below are performed for Hydrogen gas, with gas constant \(R_s = 4124.2\, \si{J / kg K } \). For all three BGK-type models, we use a temperature-dependent viscosity function that is fitted using experimental data \cite{mehl2010}\begin{align}
	\mu(T) = 8.9324 \times 10^{-6} \left(\frac{T}{299.9878}\right)^{0.6998} \, \si{Pa}. 
\end{align} For the ESBGK and polyatomic ESBGK models, we set the Prandtl number to the physically-correct value of 0.69. For the ESBGK model, this yields \(\nu = -0.4493\). For the polyatomic ESBGK model, we have two free parameters, \(\theta\) and \(\nu\). We again fix the Prandtl number to 0.69, and set \(\theta = 1/174\). This is due to the interpretation of the \(\theta\) as the inverse of the relaxation collision number. These choices lead to \(\nu = -0.4519\), for the polyatomic ESBGK model.

We compare are results to (polyatomic) DSMC simulations performed with OpenFOAM v2406. For simulations in which the Knudsen number is small and DSMC simulations are infeasible, we compare our results to simulations performed with Simcenter STAR-CCM+ 2506 Build 20.04.007. Star-CCM is a commercial computational fluid dynamics tool that implements a steady-state Navier-Stokes solver.

The remainder of this section is organised as follows. In section \ref{section:convergenceTest}, we perform a convergence analysis to verify the order of the semi-Lagrangian scheme. Section \ref{section:FourierCouette} examines the Fourier and Couette flow test cases \cite{gallis2011}, which use simplified geometries and boundary conditions to isolate the heat-conduction and shear-stress behavior of the models. We perform simulations across the free-molecular, transition, slip, and continuum regimes, and compare the results with DSMC and analytical solutions. In Section \ref{section:orifice}, we consider a more challenging orifice geometry and benchmark our method against DSMC and Navier–Stokes simulations. Finally, Section \ref{section:movingOrifice} investigates the same orifice configuration with a time-dependent opening.

\subsection{Convergence test}
\label{section:convergenceTest}
In this section, we perform a non-physical simulation to verify the accuracy of our method. Since the method is the same for all three BGK models, we perform this simulation with the BGK model only. Consider a two-dimensional periodic domain \([-1, 1]^2\), with a gas with gas constant \(R_s = 1\) and fixed relaxation time \(\tau = 1 \times 10^{-1}\). Initially, the gas is in local thermodynamic equilibrium with mean velocity \(U_0(x, y) = \left(U_0, U_0, 0\right)^\top\) with 
\begin{align}
	U_0 = \frac{1}{10} \left[ \exp\!\left(-\left(10 \sqrt{(x - 0.2)^2 + y^2} - 1\right)^2\right)
		- 2\,\exp\!\left(-\left(10 \sqrt{(x + 0.2)^2 + y^2} - 1\right)^2\right) \right]. 
\end{align} The initial temperature and pressure are set to unity. The velocity variable is discretised uniformly on the domain \([-1, 1]^2\) using \(30^2\) points. The simulation is performed up to time \(1 \times 10^{-1}\) with CFL number \(0.45\) such that only the spatial discretisation error is visible. The simulation is performed with the first-order in time method from Section \ref{section:SLFirstOrder} and with the second-order in time method from Section \ref{section:SLSecondOrder}, for increasingly fine grids. In figure \ref{fig:convergence}, we report the \(L_1\) error, computed with the higher-order method on a grid with 624 grid points. We observe that both methods achieve their order as advertised. In the following test cases, simulations are performed with the higher-order method. 
\begin{figure}
	\centering
	\includegraphics[width=0.5\linewidth]{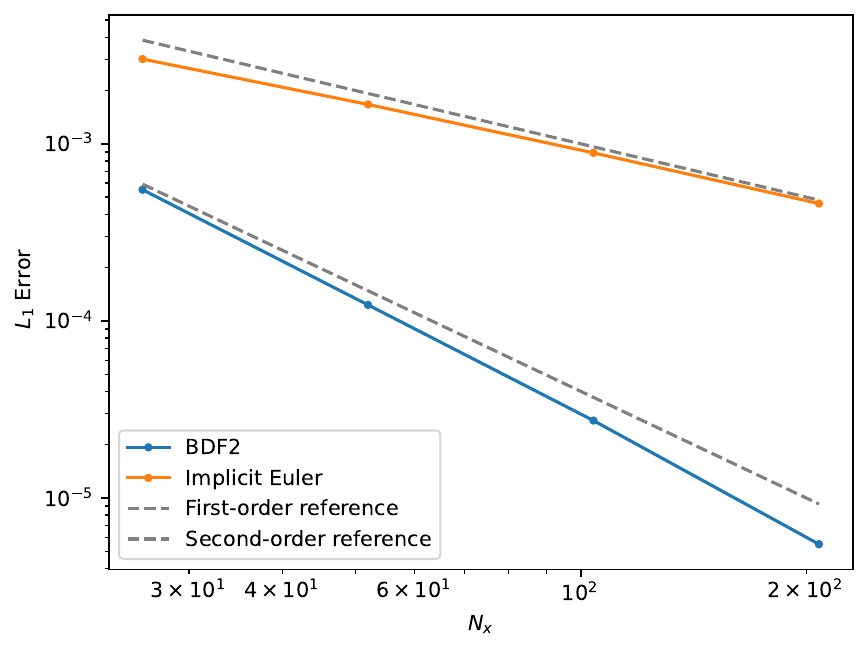}
	\caption{Convergence in the \(L_1\)-norm of the first-order scheme from Section \ref{section:SLFirstOrder} and the second-order scheme from Section \ref{section:SLSecondOrder}. }
	\label{fig:convergence}
\end{figure}

\subsection{Fourier and Couette flow}
\label{section:FourierCouette}
\newcommand{\Height}{2.5} 
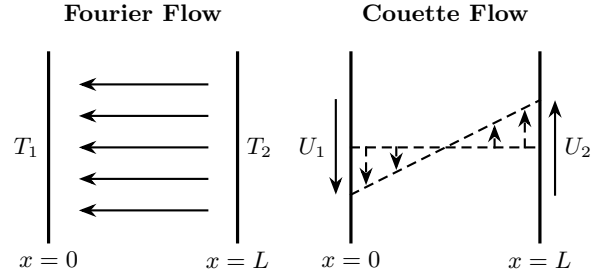
\begin{figure}[htb]
	
	\centering
	\begin{tikzpicture}[
		wall/.style={very thick},
		flowarrow/.style={->, thick},
		stressarrow/.style={->, thick, dashed},
		>=Stealth,
		font=\small,
		line cap=rect,
		line join=round
		]
		
		% =========================
		% Fourier Flow
		% =========================
		\begin{scope}
			% Walls
			\draw[wall] (0,0) -- (0,\Height);
			\draw[wall] (2.5,0) -- (2.5,\Height);
			
			% Heat flux arrows
			\foreach \i in {1,...,5} {
				\pgfmathsetmacro{\y}{\i*\Height/6}
				\draw[flowarrow] (2.1,\y) -- (0.4,\y);
			}
			
			% Labels
			\node[left]  at (0,0.5*\Height) {$T_1$};
			\node[right] at (2.5,0.5*\Height) {$T_2$};
			
			\node[below] at (0,0) {$x=0$};
			\node[below] at (2.5,0) {$x=L$};
			
			\node[above] at (1.25,\Height+0.3) {\textbf{Fourier Flow}};
		\end{scope}
		
		% =========================
		% Couette Flow
		% =========================
		\begin{scope}[xshift=4cm]
			% Walls
			\draw[wall] (0,0) -- (0,\Height);
			\draw[wall] (2.5,0) -- (2.5,\Height);
			
			% Velocity arrows at walls
			\draw[flowarrow] (-0.2,0.75*\Height) -- (-0.2,0.25*\Height);
			\draw[flowarrow] (2.7,0.25*\Height) -- (2.7,0.75*\Height);
			
			% Velocity labels
			\node[left]  at (-0.2,0.5*\Height) {$U_1$};
			\node[right] at (2.7,0.5*\Height) {$U_2$};
			
			% Linear velocity profile
			\draw[thick, dashed] (0, 0.25*\Height) -- (2.5, 0.75*\Height);
			\draw[thick, dashed] (0, 0.5*\Height) -- (2.5, 0.5*\Height);
			
			% Shear stress arrows
			\draw[stressarrow] (0.2,0.5*\Height) -- (0.2,0.3*\Height);
			\draw[stressarrow] (0.6,0.5*\Height) -- (0.6,0.38*\Height);
			\draw[stressarrow] (2.3,0.5*\Height) -- (2.3,0.7*\Height);
			\draw[stressarrow] (1.9,0.5*\Height) -- (1.9,0.62*\Height);
			
			% Labels
			\node[below] at (0,0) {$x=0$};
			\node[below] at (2.5,0) {$x=L$};
			
			\node[above] at (1.25,\Height+0.3) {\textbf{Couette Flow}};
		\end{scope}
		
	\end{tikzpicture}
	\caption{Illustration of the Fourier and Couette flow.}
	\label{fig:FourierCouette}
\end{figure}

\begin{table}[htb]
	\centering
	\begin{tabular}{||c c c c||}
		\hline
		$L$ & $T_{\text{ref}}$ & $\Delta T$ & $\Delta U$ \\
		\hline\hline
		$1\,\si{mm}$ & $273.15\,\si{K}$ & $70\,\si{K}$ & $100\,\si{m/s}$ \\
		\hline
	\end{tabular}
	\caption{Simulation parameters for the Fourier and Couette flow simulations.}
	\label{tab:FCValues}
\end{table}
We consider a square 2D domain \(\left[0, L\right]^2\) with periodic boundaries in the vertical direction and diffuse-reflective walls on the sides, see Figure \ref{fig:FourierCouette}. In the case of a Fourier flow, the left and right walls are stationary and have temperatures \(T_1 = T_{ref} - \frac{\Delta T}{2}\) and \(T_2 = T_{ref} + \frac{\Delta T}{2}\). In the case of the Couette flow, the left and right walls have a temperature \(T_{ref}\), and velocities \(\frac{-\Delta U}{2}\) and \(\frac{\Delta U}{2}\) respectively. All parameters are given in Table \ref{tab:FCValues}. We perform simulations for ten pressure values: \(0.1\, \si{Pa}\), \(0.26\, \si{Pa}\), \(0.66\, \si{Pa}\), \(1.71\, \si{Pa}\), \(4.41\, \si{Pa}\), \(11.35\, \si{Pa}\), \(29.24\, \si{Pa}\), \(75.33\, \si{Pa}\), \(194.08\, \si{Pa}\) and \(500\, \si{Pa}\). This yields Knudsen numbers between \(10^2\) and \(10^{-2}\). Below, we investigate the theoretical temperature profile and heat flux for the Fourier flow, and the theoretical velocity profile and shear stress for the Couette flow. After, we compare with numerical simulations of the BGK models and DSMC.  

\textbf{Continuum regime:} The temperature profile for the Fourier flow, and the velocity profile for the Couette flow are \cite{gallis2011} \begin{align}
	T_c &= \sqrt{T_1^2 + \left(T_2^2 - T_1^2\right)\frac{x}{L}}, \label{eq:Tc} \\
	U_c &= \Delta U \left(\frac{x}{L} - \frac{1}{2}\right). \label{eq:Uc}
\end{align} Then, the shear stress, which is the \((1, 2)\)-element of the stress tensor \(\Stensor\), is given by \begin{align} 
	\tau_{xy, c} &= -\mu \left(\partial_x U_y + \partial_y U_x \right) = -\mu \frac{\Delta U}{L}.
\end{align} For the Fourier flow, the heat flux is given by \begin{align}
	q_{x, c} &= -\frac{5 + \delta}{2} \mu \left(1 - \nu + \theta \nu\right) R_s \nabla_x T(x) = -\frac{5 + \delta}{2} \mu \left(1 - \nu + \theta \nu\right) R_s \frac{\left(T_2^2 - T_1^2\right)}{2L\sqrt{T_1^2 + \left(T_2^2 - T_1^2\right)\frac{x}{L}}}. \label{eq:heatFluxC}
\end{align} When plotting the heat flux, we will do so at \(x = \frac{L}{2}\). 

\textbf{Free molecular regime:} The distribution function of the gas inside the channel is given by two half Maxwellians propagating from the left and right wall \begin{align}
	f(x, v, t, I) = \begin{cases}
		\mathcal{M}(\rho_1, U_1, T_1) \quad \text{if } v_1 > 0, \\
		\mathcal{M}(\rho_2, U_2, T_2) \quad \text{if } v_1 < 0, 
	\end{cases} \label{eq:FMDistribution}
\end{align} where \(\mathcal{M}(\rho, U, T)\) is the \(\mathcal{M}\) is the Maxwellian distribution with density \(\rho\), mean velocity \(U\) and temperature \(T\). At the left and right wall, we have that the flux across the wall should be zero, therefore, \begin{align}
	\int_{v_x > 0} v \cdot n  \mathcal{M}(\rho_1, U_1, T_1) dv dI + \int_{v_x < 0} v \cdot n  \mathcal{M}(\rho_2, U_2, T_2) dv dI = 0,
\end{align} which simplifies to, \begin{align}
	\rho_1 \sqrt{T_1} = \rho_2 \sqrt{T_2}. \label{eq:FMCondition}
\end{align} The theoretical temperature and velocity profile for the Fourier and Couette flow are \cite{gallis2011} \begin{align}
	T_{fm} &= \sqrt{T_1 T_2}, \label{eq:fmT} \\
	U_{fm} &= 0 \label{eq:fmU}.
\end{align} Using Equation \eqref{eq:FMDistribution} and \eqref{eq:FMCondition}, we find for the shear stress \begin{align}
	\tau_{xy, fm} &= \int_{\mathbb{R}^3 \cross \mathbb{R}^+} v_1 v_2 f(t, x, v, I) dv dI = \frac{-p \Delta U}{\sqrt{2 \pi R_s T}}. 
\end{align} For the heat flux in the horizontal direction, we find \begin{align}
	q_{x, fm} &= \int_{\mathbb{R}^3 \cross \mathbb{R}^+} v_1 \left(\frac{|v|^2}{2} + I^{2/\delta}\right) f(t, x, v, I) dv dI = \sqrt{\frac{2 R_s T_1}{\pi}} p \left(\frac{\delta}{4}+1\right)\left(1 - \frac{T_r}{T_l}\right). \label{eq:heatFluxFM}
\end{align}
We perform the Fourier and Couette flow simulation for all ten pressure values with the BGK, ESBGK and polyatomic ESBGK models. The simulations are performed with \(210\) grid points in the horizontal direction and CFL number \(0.45\). The velocity variable is discretised on \([-6 \sqrt{R_s T_2}, 6 \sqrt{R_s T_2}]^2\) with \(N_v^2 = 40^2\) points. Initially, the gas is stationary and in thermodynamic equilibrium with temperature \(\left( T_1 + T_2 \right)/2\) and a constant pressure profile. We perform time-dependent simulations until equilibrium has been reached. We compare the results due to the BGK models with DSMC simulations. In Figure \ref{fig:CouetteU}, the velocity profile for the Couette flow is plotted for several pressure values. All models yield profiles that match the theoretical values in the continuum regime \eqref{eq:Uc} and free molecular regime \eqref{eq:fmU} because those results are independent of the collision model. Also for intermediate values of the pressure, all models yield very similar velocity profiles to DSMC. The departure from linearity in the continuum limit, is due to temperature-dependent viscosity \cite{gallis2011}. The temperature profile for the Fourier flow case is plotted in Figure \ref{fig:FourierT}. The theoretical limits were derived in \eqref{eq:Tc} and \eqref{eq:fmT}. As expected, all models yield the same results in the continuum limit and in the free molecular regime. The departure from linearity in the continuum limit, is due to temperature-dependent heat conduction \cite{gallis2011}. For intermediate pressures, the standard BGK model gives the worst results. This is because the relaxation time in the BGK model was chosen to match the viscosity coefficient in the continuum regime. The ESBGK and diatomic ESBGK model yield results close to DSMC, although for pressure \(0.441\, \si{Pa}\), we do observe a significant difference with the DSMC result. This is to be expected, as for this pressure, the gas is in the transition regime, and the BGK approximation is known to provide a crude approximation. In conclusion, the velocity profile for the Couette flow is well approximated by all three BGK models. It is more difficult to obtain a temperature field in the Fourier flow close to that of DSMC, however, the ESBGK model and diatomic ESBGK obtain good results.
\begin{figure}[htbp]
	\centering
	\begin{subfigure}{0.49\textwidth}
		\centering
		\includegraphics[width=\linewidth]{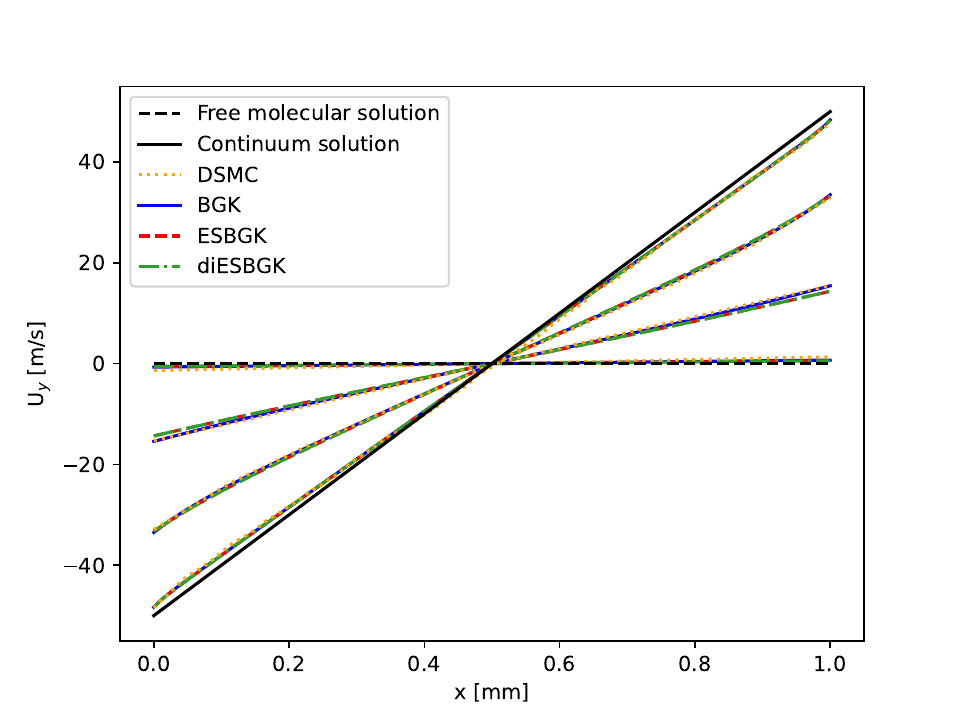}
		\caption{Velocity profile for the Couette flow.}
		\label{fig:CouetteU}
	\end{subfigure}
	\hfill
	\begin{subfigure}{0.49\textwidth}
		\centering
		\includegraphics[width=\linewidth]{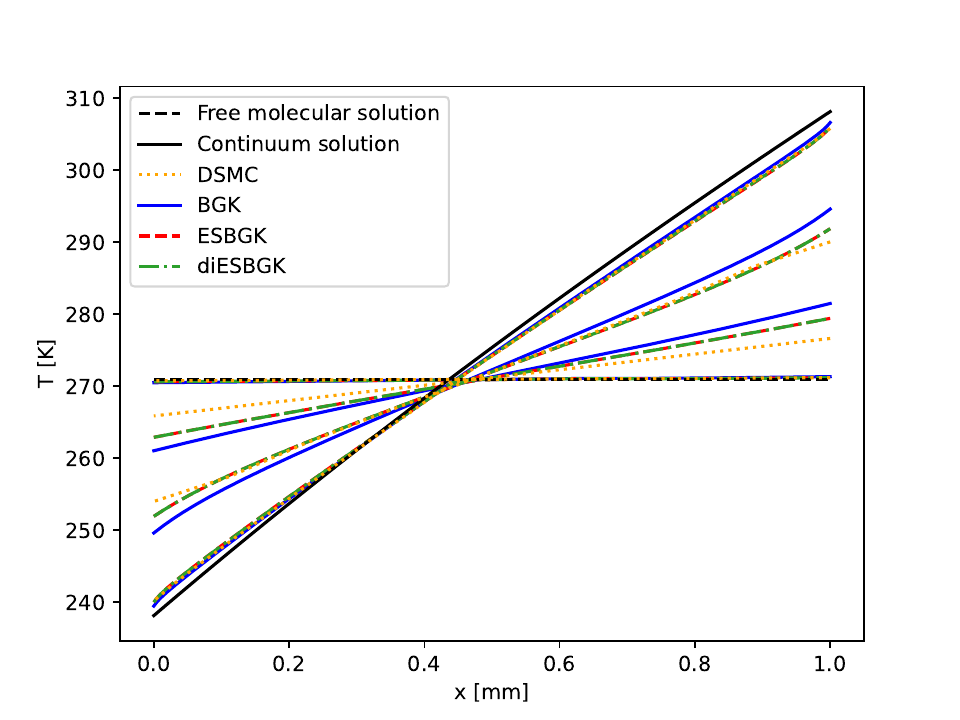}
		\caption{Temperature profile for the Fourier flow.}
		\label{fig:FourierT}
	\end{subfigure}
	\caption{Temperature and velocity along a horizontal slice for several pressure values \(0.1\,\si{Pa}, 4.41\,\si{Pa}, 29.24\,\si{Pa} \) and \(500\,\si{Pa}\) (horizontal line to line with positive slope). The associated Knudsen numbers are \(2 \times 10^{2}\), \(4.6 \times 10^{0}\), \(7.0 \times 10^{-1}\) and \(4.1 \times 10^{-2}\).}
	\label{fig:FCTempVelocity}
\end{figure}

In Figure \ref{fig:shearStress}, the shear stress for the Couette flow is plotted as a function of the pressure. All three BGK models obtain the same result by design, which also agrees well with the DSMC result. In Figure \ref{fig:heatFlux}, the heat flux is plotted for the Fourier flow. We have included the theoretical limits of the heat flux for a monoatomic gas and a polyatomic gas, which were derived in Equation \eqref{eq:heatFluxC} and \eqref{eq:heatFluxFM}. We see that the ESBGK and polyatomic ESBGK models approach the correct values in the continuum and free molecular limit. The BGK model does not approach the correct value in the continuum limit due to its limitation that the Prandtl number is fixed to one. It can either produce the correct shear stress, or the correct heat conduction, but not both. The polyatomic ESBGK result agrees well with the DSMC result. This test also shows the necessity of a polyatomic model when the heat conduction is an important quantity of interest. 
\begin{figure}[htbp]
	\centering
	\begin{subfigure}{0.49\textwidth}
		\centering
		\includegraphics[width=\linewidth]{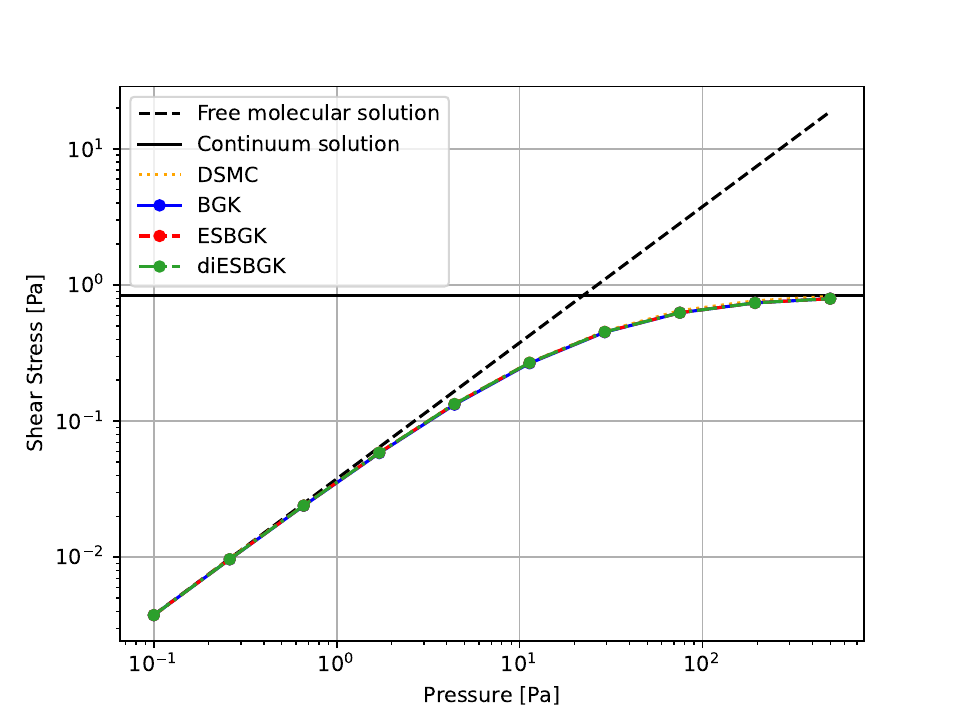}	
		\caption{Shear stress for the Couette flow.}
		\label{fig:shearStress}
	\end{subfigure}
	\hfill
	\begin{subfigure}{0.49\textwidth}
		\centering
		\includegraphics[width=\linewidth]{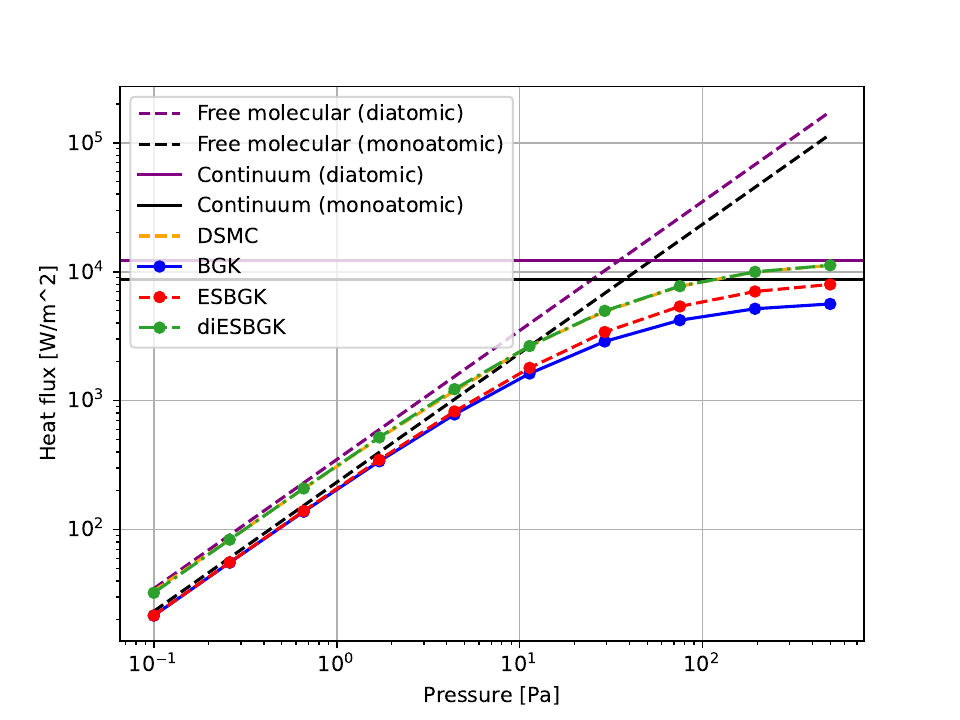}
		\caption{Heat flux for the Fourier flow.}
		\label{fig:heatFlux}
	\end{subfigure}
	\caption{Heat flux and shear stress for the Fourier and Couette flow as a function of the pressure.}
	\label{fig:HeatShearStress}
\end{figure}

\begin{remark}
	The Fourier and Couette flow simulations were carried out for multiple values of $\theta$ using the polyatomic ESBGK model. Since no significant differences were observed, we use $\theta = 1/174$ for the remainder of the simulations.
\end{remark}

\subsection{Orifice simulation}
\label{section:orifice}
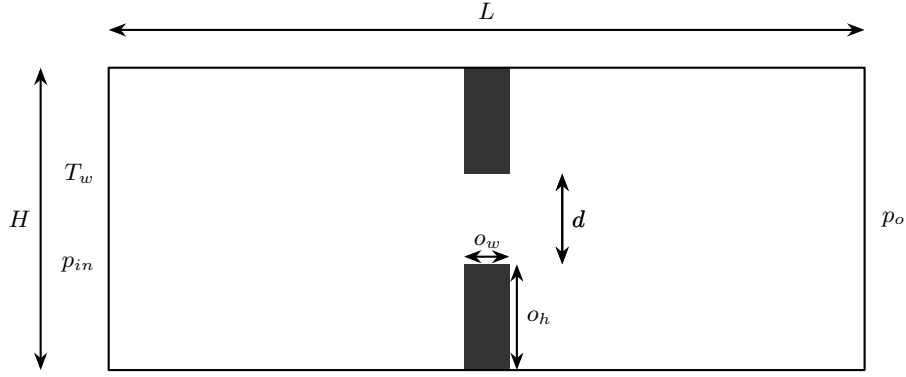
\begin{figure}[t]
	\centering
	
	% Adjustable geometry parameters
	\newcommand{\LOriface}{10}   % total length
	\newcommand{\HOriface}{4}    % channel height
	\newcommand{\dO}{1.2} % orifice diameter
	\newcommand{\wO}{0.6} % orifice thickness
	
	\begin{tikzpicture}[
		>=Stealth,
		thick,
		line cap=rect,
		line join=round,
		font=\small
		]
		
		% =========================
		% Outer channel
		% =========================
		\draw (0,0) rectangle (\LOriface,\HOriface);
		
		% =========================
		% Orifice block
		% =========================
		\filldraw[white]
		({\LOriface/2 - \wO/2},{(\HOriface-\dO)/2})
		rectangle
		({\LOriface/2 + \wO/2},{(\HOriface+\dO)/2});
		
		% Oriface
		
		\fill[black, opacity=0.8] ({\LOriface/2 - \wO/2},0) rectangle ({\LOriface/2 + \wO/2},{(\HOriface-\dO)/2});
		
		\fill[black, opacity=0.8] ({\LOriface/2 - \wO/2},{(\HOriface+\dO)/2}) rectangle ({\LOriface/2 + \wO/2},\HOriface);
		
		% =========================
		% Dimension arrows
		% =========================
		
		% Length L
		\draw[<->] (0,\HOriface+0.5) -- (\LOriface,\HOriface+0.5);
		\node at (\LOriface/2,\HOriface+0.75) {$L$};
		
		% Height H (left)
		\draw[<->] (-0.9,0) -- (-0.9,\HOriface);
		\node[left] at (-0.9,{\HOriface/2}) {$H$};
		
		% Orifice diameter d
		\draw[<->] ({\LOriface/2 + 1.0},{(\HOriface-\dO)/2})
		-- ({\LOriface/2 + 1.0},{(\HOriface+\dO)/2});
		
		\node[right] at ({\LOriface/2 + 1.0},{\HOriface/2}) {$d$};
		
		% Orifice diameter d
		\draw[<->] ({\LOriface/2 + 1.0},{(\HOriface-\dO)/2}) -- ({\LOriface/2 + 1.0},{(\HOriface+\dO)/2}); 
		
		\node[right] at ({\LOriface/2 + 1.0},{\HOriface/2}) {$d$};
		
		% Oriface height
		\draw[<->] ({\LOriface/2 + \wO/2 + 0.1},0) -- ({\LOriface/2 + \wO/2 + 0.1},{(\HOriface-\dO)/2}); 
		
		\node[right] at ({\LOriface/2 + \wO/2 + 0.1},{(\HOriface-\dO)/4}) {$o_h$};
		
		% Oriface width
		\draw[<->] ({\LOriface/2 - \wO/2},{(\HOriface-\dO)/2 + 0.1}) -- ({\LOriface/2 + \wO/2},{(\HOriface-\dO)/2 + 0.1}); 
		
		\node[right] at ({\LOriface/2 - \wO/2},{(\HOriface-\dO)/2 + 0.3}) {$o_w$};
		
		% =========================
		% Boundary labels
		% =========================
		\node[left, align=right] at (-0.05,{0.65*\HOriface}) {$T_{w}$};
		\node[left, align=right] at (-0.05,{0.35*\HOriface}) {$p_{in}$};
		
		\node[right] at (\LOriface+0.1,{\HOriface/2}) {$p_o$};
		
	\end{tikzpicture}
	
	\caption{Basic structure and boundary conditions for the orifice flow.}
	\label{fig:Oriface}
\end{figure}

\begin{table}[htb]
	\centering
	\begin{tabular}{||c c c c c c c c||}
		\hline
		$\Delta x$ & $\Delta y$ & $L$ & $H$ & $o_w$ & $o_h$ & \(d/H\) & $T_w$ \\
		\hline\hline
		\(6.1350 \times 10^{-1} \, \si{\mu m} \) & \(6.0241 \times 10^{-1} \, \si{\mu m}\) & \(0.5\, \si{mm} \)  & \(0.1\, \si{mm}\) & \( 12 \Delta x\) & \( 71 \Delta y\) & \(14.46 \%\) & \(300 \, \si{K} \) \\
		\hline
	\end{tabular}
	\caption{Simulation parameters for the Fourier and Couette flow simulations. The quantities \(\Delta x\) and \(\Delta y\) are the grid size of the Cartesian grid for the semi-Lagrangian method.}
	\label{tab:OrifaceValues}
\end{table}

We consider a rectangular 2D domain \(\left[0, L\right] \cross \left[0, H\right] \) with an orifice of width \(o_w\) and height \(o_h\), similar to \cite{wang2004}. Hydrogen gas flows in the domain from the left with pressure \(p_{in}\), and temperature \(T_{w}\). The gas exists the domain on the right, where the pressure is set to \(p_o\). The top and bottom walls, as well as the walls of the orifice, have temperature \(T_{w}\). To reduce the computational time, we impose a specular boundary condition along the horizontal centreline of the geometry and only simulate the bottom half of the geometry. All relevant simulation parameters are given in table \ref{tab:OrifaceValues}. We perform the simulation with two outflow pressures \(p_{o}\), \(50 \, \si{kPa}\) and \(2.5 \, \si{kPa}\). For both outflow pressures, we perform the simulation with four inflow pressures: \(2 p_{o}\), \(7p_o/4\), \(6p_o/4\), \(5p_o/4 \). For the high and low outflow pressure and \(p_{in}=2p_{o}\), we obtain a Knudsen number of \(4 \times 10^{-3}\) and \(9 \times 10^{-2}\) at the outflow boundary. At the orifice, the Knudsen number goes up to \(3 \times 10^{-2}\) and \(6 \times 10^{-1}\). In all simulations, the CFL number is set to 0.49. The simulations are run until time \(2 \times 10^{-6}\), at which point the solution as converged to a steady state. The translational temperature, mean velocity and pressure along centreline of the orifice for \(p_{o} = 50 \, \si{kPa}\) are plotted in figure \ref{fig:orifice1}.

Before the orifice, the temperature, pressure and velocity are nearly constant. At the orifice, the pressure decreases sharply, overshoots the outflow pressure and then slowly reaches outflow pressure. Due to the constriction, the gas accelerates through the orifice, after which the gas purges into the chamber behind the orifice. Due to the increased velocity at the orifice, some of the internal energy of the gas is converted to kinetic energy. As a result, the translational temperature of the gas decreases at the orifice. The decrease in pressure and temperature, and the increase in velocity, increases with the inflow pressure. We also observe a recirculation of the flow behind the orifice, see figure \ref{fig:orificeQuiver1}. The BGK models agree well pressure and velocity profile obtained with starccm. We attribute the difference in the temperature profile to the different model (BGK-type vs Navier-Stokes), different grid (Cartesian vs unstructured and locally refined) and the method (higher-order vs first-order). Although, it is worth noting that of all three BGK models, the polyatomic ESBGK model yields results closest to the starcmm profile. 
\begin{figure}
	\centering
	\begin{subfigure}[b]{0.49\textwidth}
		\centering
		\includegraphics[width=\linewidth]{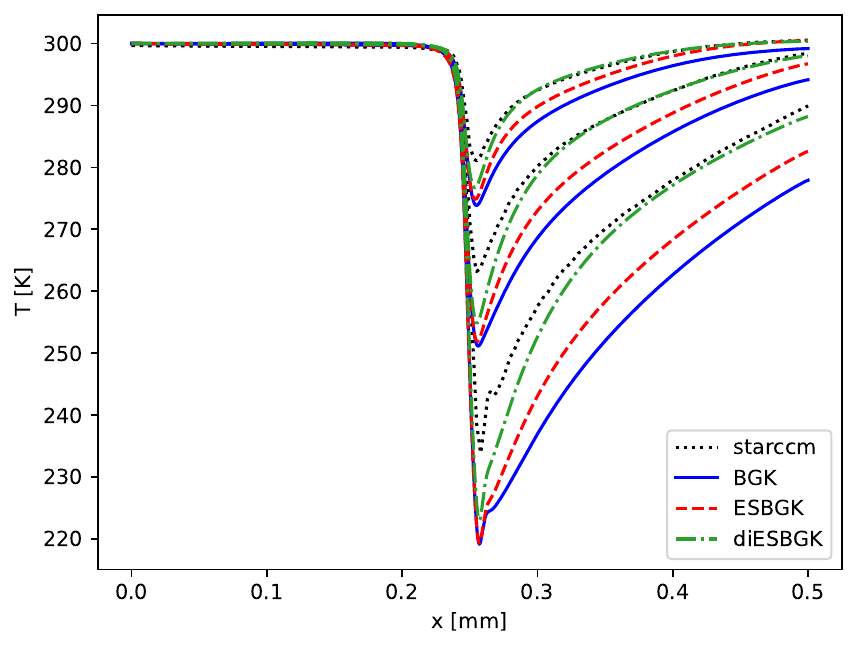}
		\caption{Translational temperature \(T_{tr}\).}
		\label{fig:OrificeTemperature1}
	\end{subfigure}
	\begin{subfigure}[b]{0.49\textwidth}
		\centering
		\includegraphics[width=\linewidth]{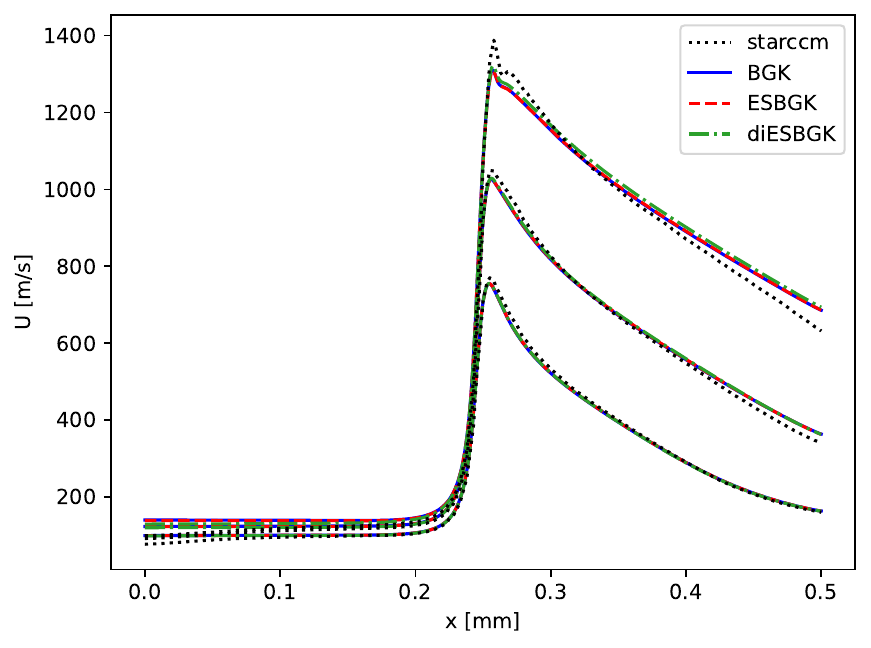}
		\caption{Mean velocity \(U\).}
		\label{fig:OrificeVelocity1}
	\end{subfigure}
	\begin{subfigure}[b]{0.49\textwidth}
		\centering
		\includegraphics[width=\linewidth]{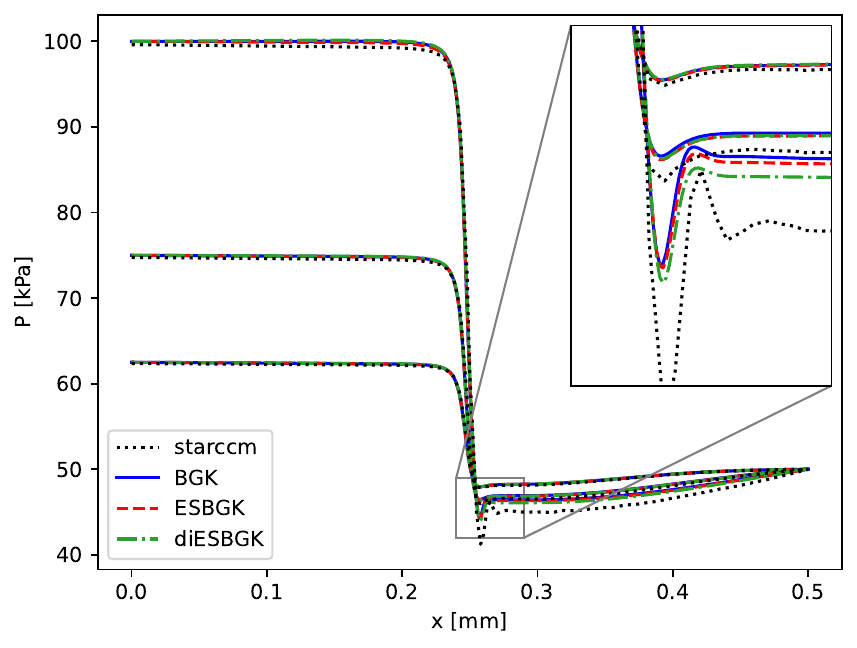}
		\caption{Pressure \(p = \rho R_s T_{tr}\).}
		\label{fig:OrificePressure1}
	\end{subfigure}
	\caption{The orifice simulation for outflow pressure \(p_o = 50 \, \si{kPa}\), and inflow pressures \(100 \, \si{kPa}\), \(75 \, \si{kPa}\) and \(62.5 \, \si{kPa}\).}
	\label{fig:orifice1}
\end{figure}

\begin{figure}
	\centering
	\includegraphics[width=\linewidth]{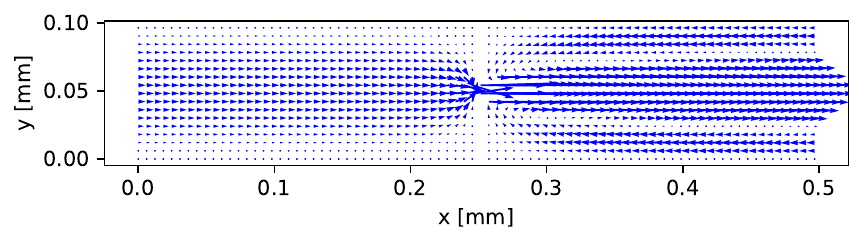}
	\caption{The horizontal velocity of the diatomic ESBGK model for \(p_o = 50 \, \si{kPa}\) and \(p_{in} = 2 p_o\). A vortex is observed behind the orifice.}
	\label{fig:orificeQuiver1}
\end{figure}

The translational temperature, mean velocity and pressure along centreline of the orifice for \(p_{o} = 2500\, \si{Pa}\) are plotted in figure \ref{fig:orifice2}. For this pressure, the gas is in the transition regime, and comparing with Navier-Stokes is fruitless. Instead, we compare our result to DSMC simulations. Compared to the case with large outflow, we immediately observe a more smooth pressure, temperature and mean velocity profiles, as one would expect from a larger Knudsen number flow. The pressure is well approximated by the BGK models. The BGK models overestimate the velocity along the centre line slightly. In the temperature plot, we observe significant statistical noise in the DSMC simulations. For the lower pressure simulations, we do not observe the recirculation in these simulation as is observed for the larger pressure in figure \ref{fig:orificeQuiver1}.
\begin{figure}
	\centering
	\begin{subfigure}[b]{0.49\textwidth}
		\centering
		\includegraphics[width=\linewidth]{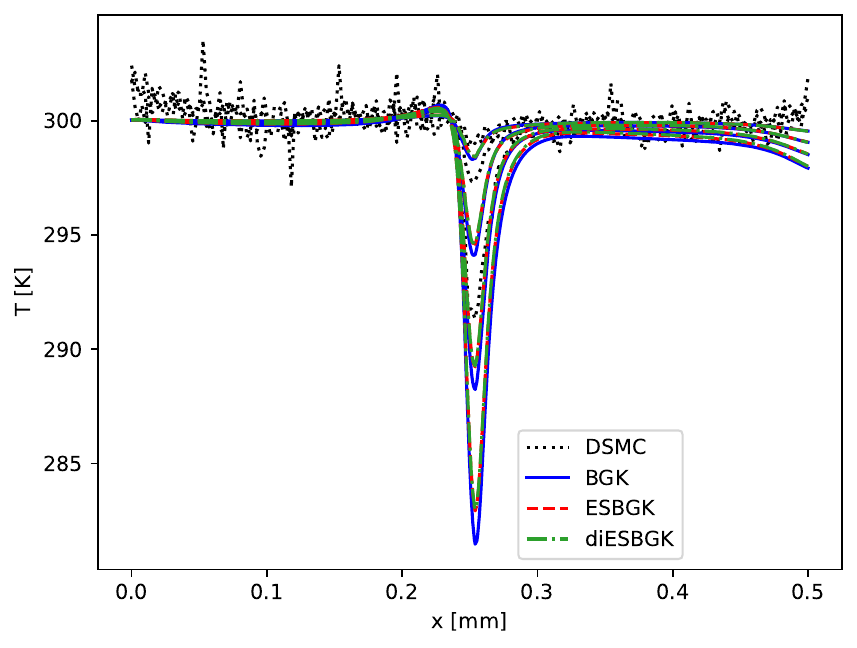}
		\caption{Translational temperature \(T_{tr}\).}
		\label{fig:OrificeTemperature2}
	\end{subfigure}
	\begin{subfigure}[b]{0.49\textwidth}
		\centering
		\includegraphics[width=\linewidth]{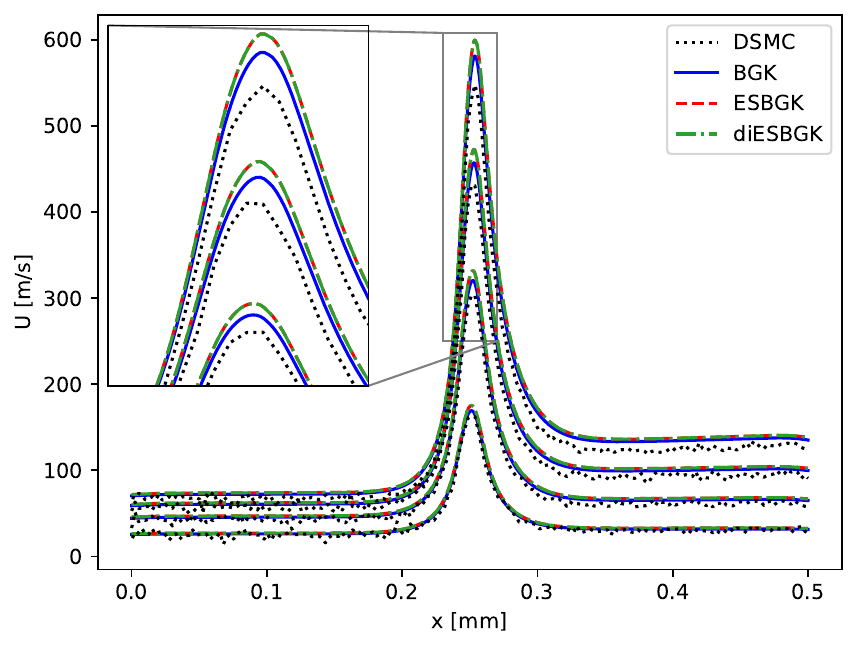}
		\caption{Mean velocity \(U\).}
		\label{fig:OrificeVelocity2}
	\end{subfigure}
	\begin{subfigure}[b]{0.49\textwidth}
		\centering
		\includegraphics[width=\linewidth]{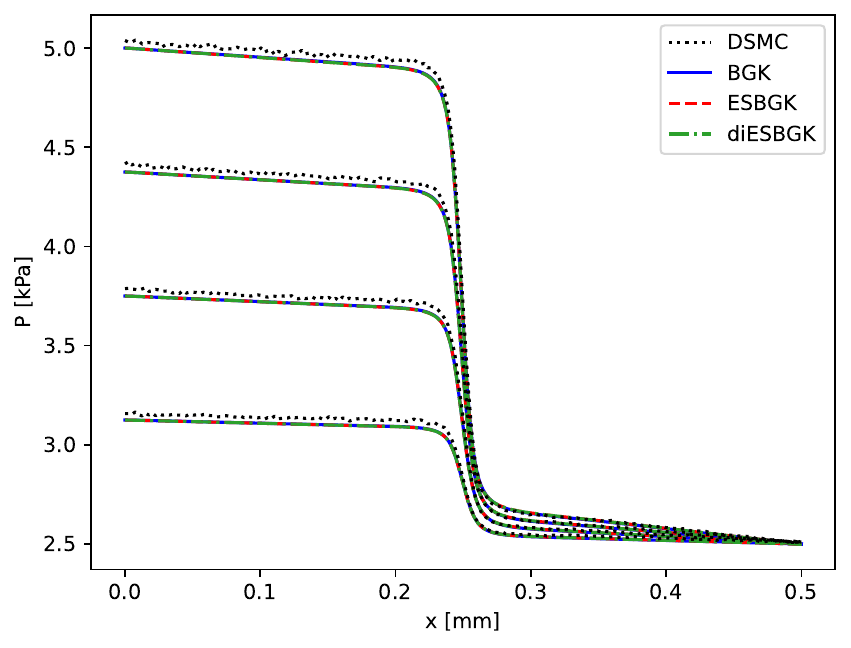}
		\caption{Pressure \(p = \rho R_s T_{tr}\).}
		\label{fig:OrificePressure2}
	\end{subfigure}
	\caption{The orifice simulation for outflow pressure \(p_o = 2500\, \si{Pa}\), and inflow pressures \(5000 \, \si{Pa}\), \(4375 \, \si{Pa}\), \(3750 \, \si{kPa}\) and \(3125 \, \si{Pa}\).}
	\label{fig:orifice2}
\end{figure}

In figure \ref{fig:orificeMassFlowRate}, the mass flow rate is plotted as a function of the pressure for the two outflow pressures. In both cases, all three BGK models yield very similar results and agree with the DSMC/Navier-Stokes simulations.
\begin{figure}
	\centering
	\begin{subfigure}[b]{0.49\textwidth}
		\centering
		\includegraphics[width=\linewidth]{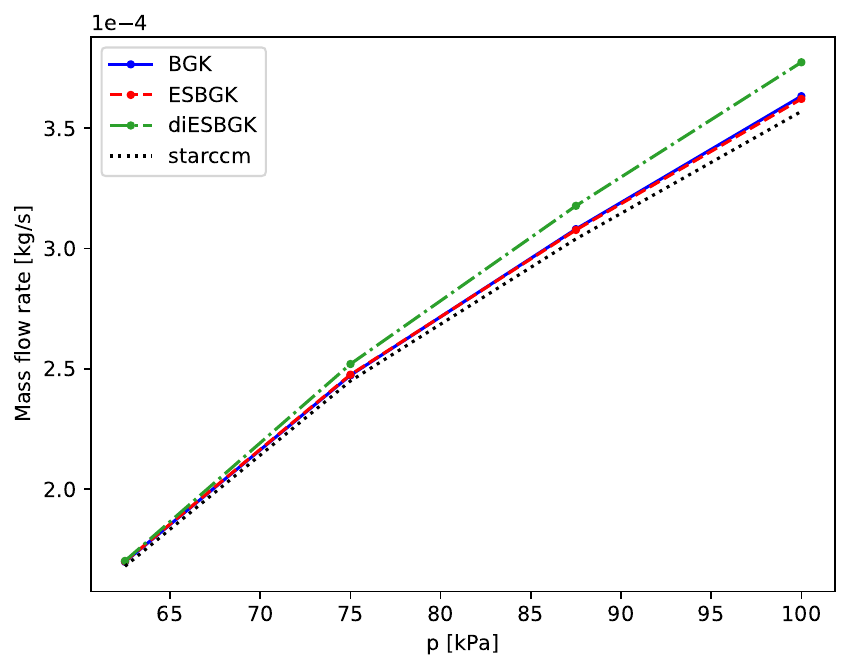}
		\caption{Outflow pressure \(p_o = 50000\, \si{Pa}. \)}
		\label{fig:OrificeMassFlowRate1}
	\end{subfigure}
	\begin{subfigure}[b]{0.49\textwidth}
		\centering
		\includegraphics[width=\linewidth]{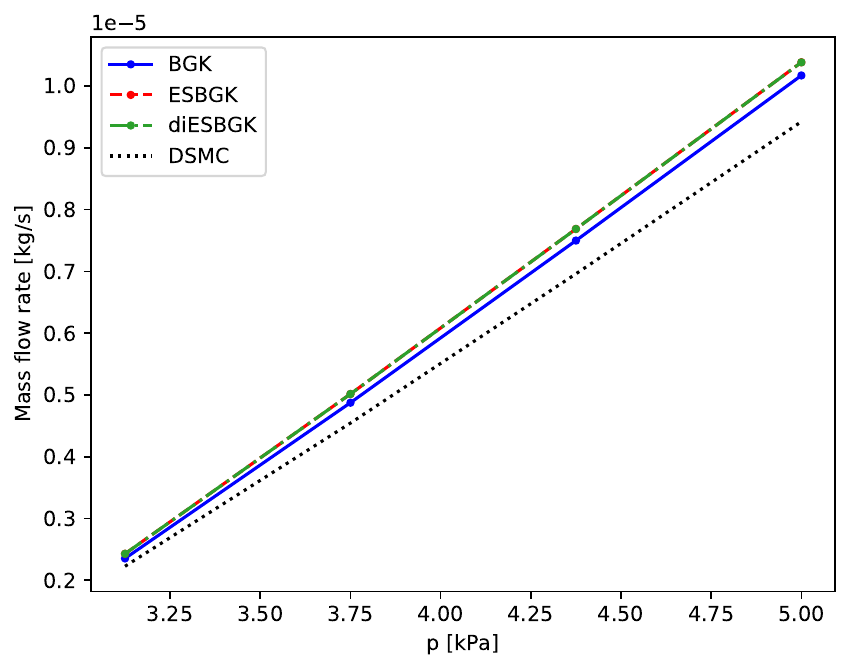}
		\caption{Outflow pressure \(p_o = 2500\, \si{Pa}. \).}
		\label{fig:OrificeMassFlowRate2}
	\end{subfigure}
	\caption{}
	\label{fig:orificeMassFlowRate}
\end{figure}
\clearpage

\subsection{Moving orifice simulation}
\label{section:movingOrifice}
We consider the same orifice boundary as before, see figure \ref{fig:Oriface}, but make the height of the oriface time-dependent, \begin{align}
	o_h(t) = \begin{cases}
		2 \Delta y \text{ if } t < 2 \, \si{\mu s} \\
		2 \Delta y + \frac{H}{2} \frac{1}{3} + \frac{H}{2} \frac{1}{3} \sin \left(  \frac{2\pi}{3 \times 10^{-6}}(t - 2 \times 10^{-6}) + \frac{3\pi}{2} \right) \text{ if  } t \geq 2 \, \si{\mu s}. 
	\end{cases}
\end{align} We set the inflow pressure to \(p_{in} = 2 p_o\), and we set the outflow pressure to \(p_o = 50 000\, \si{Pa}\). The simulation is performed using the polyatomic ESBGK model up to time \(t = 1 \times 10^{-6}\). The moving boundaries are implemented using a rudimentary immersed boundary method: the solid wall advances by `snapping' it to the closest grid cell. 

We plot the pressure and horizontal velocity at several time steps in figure \ref{fig:movingOrifice}. At first, the orifice height is small compared to the size of the domain \(H\). The flow then effectively behaves like a pressure-driven Poiseuille flow, see for example \cite{meng2013, wang2004}. In this regime, for the large pressure, the horizontal velocity along a vertical slice has a parabolic profile, the pressure along the centreline decreases linearly, and the velocity along the centreline increases linearly. As the pressure decreases, slip at the boundaries and nonlinear pressure and velocity profiles indicate departure from continuum conditions. As the orifice height increases, the geometry changes to two chambers connected by a small channel. The gas then expands into the chamber behind the orifice, potentially causing recirculation behind the orifice. In Figure \ref{fig:movingOrificeTemperature3}, we observe shock diamonds behind the orifice. 

\begin{figure}	
	\centering
	\begin{subfigure}{0.45\textwidth}
		\centering
		\includegraphics[width=\linewidth]{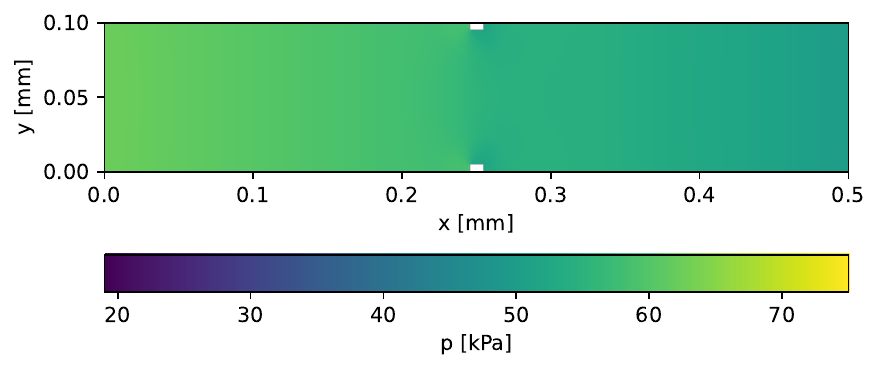}
		\caption{}
		\label{fig:movingOrificeTemperature1}
	\end{subfigure}
	\hfill
	\begin{subfigure}{0.45\textwidth}
		\centering
		\includegraphics[width=\linewidth]{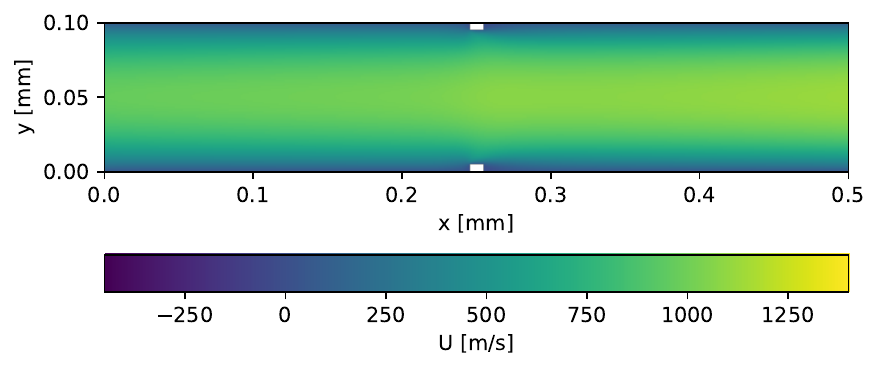}
		\caption{}
		\label{fig:movingOrificeTemperature2}
	\end{subfigure}
	
	\vspace{0.5em}
	
	\begin{subfigure}{0.45\textwidth}
		\centering
		\includegraphics[width=\linewidth]{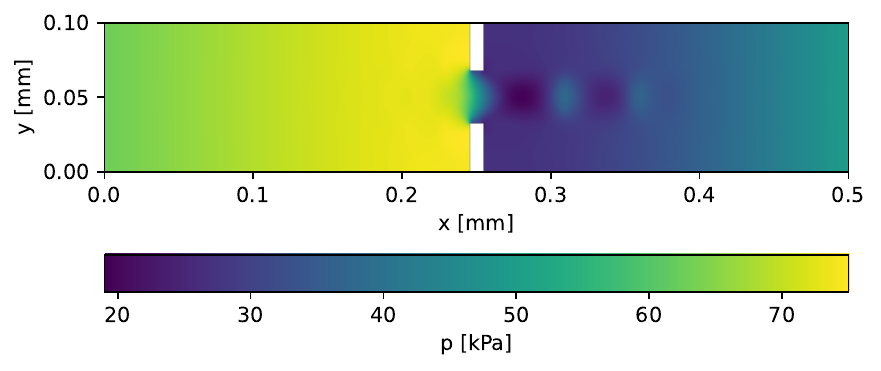}
		\caption{}
		\label{fig:movingOrificeTemperature3}
	\end{subfigure}
	\hfill
	\begin{subfigure}{0.45\textwidth}
		\centering
		\includegraphics[width=\linewidth]{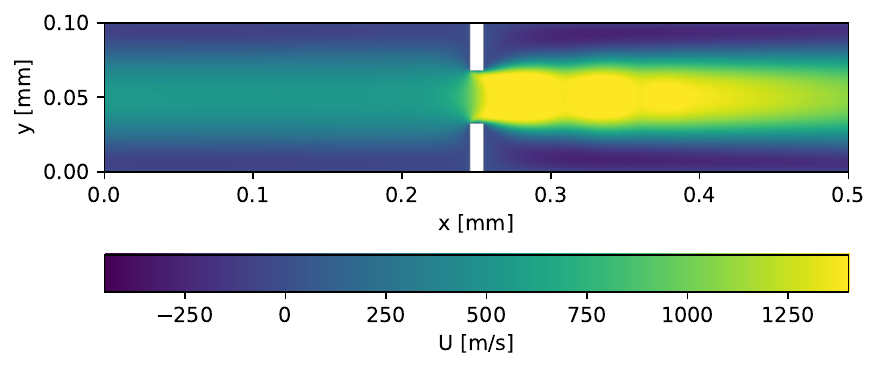}
		\caption{}
		\label{fig:movingOrificeTemperature4}
	\end{subfigure}
	
	\vspace{0.5em}
	
	\begin{subfigure}{0.45\textwidth}
		\centering
		\includegraphics[width=\linewidth]{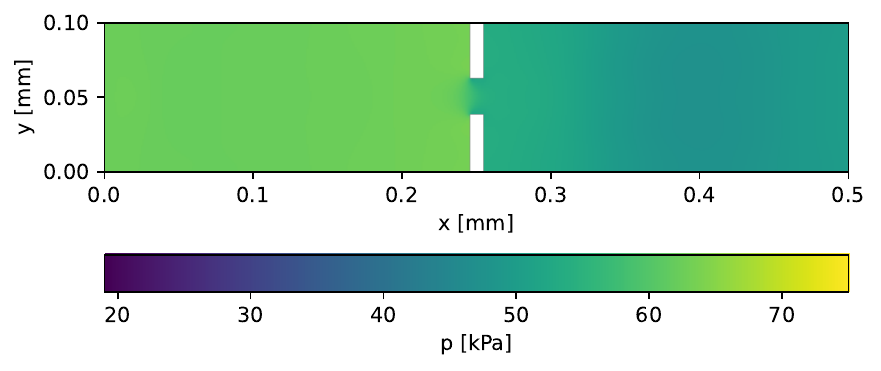}
		\caption{}
		\label{fig:movingOrificeTemperature5}
	\end{subfigure}
	\hfill
	\begin{subfigure}{0.45\textwidth}
		\centering
		\includegraphics[width=\linewidth]{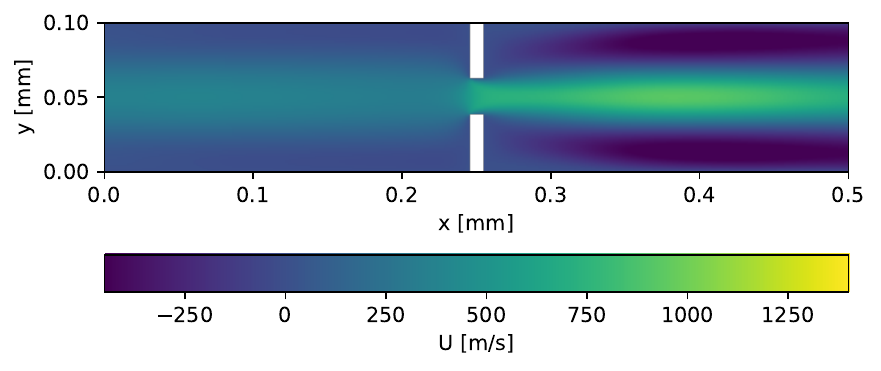}
		\caption{}
		\label{fig:movingOrificeTemperature6}
	\end{subfigure}
	
	\caption{The pressure (left) and the horizontal velocity (right) for the moving orifice simulation at times \(3.11372 \times 10^{-6}\, \si{s}\) (top), \(4.10444 \times 10^{-6}\, \si{s} \) (centre) and \(4.52904 \times 10^{-6}\, \si{s}\) (bottom).}
	\label{fig:movingOrifice}
\end{figure}

\section{Conclusion}
\label{section:Conclusion}
In this work, we developed a semi-Lagrangian numerical scheme for the polyatomic Ellipsoidal Statistical BGK (ESBGK) model of the Boltzmann equation. By combining a characteristic-based treatment of the transport term with an implicit A-stable linear multistep discretization of the relaxation operator, the proposed method eliminates the restrictive time-step constraint typically associated with kinetic transport while efficiently handling the stiffness of the collision term. Owing to the structure of the BGK operator, the implicit formulation can be reformulated into a computationally inexpensive update, resulting in a numerically stable and efficient scheme.

The method is asymptotic preserving and stiffly accurate, ensuring that in the vanishing Knudsen number limit the scheme naturally transitions to a consistent discretization of the Euler equations. Moreover, we established that the first-order version of the scheme asymptotically converges to the compressible Navier–Stokes equations with the correct transport coefficients. Suitable inflow and outflow boundary conditions for BGK-type kinetic models were also introduced, enabling the treatment of open-domain problems.

Numerical experiments, including Fourier and Couette flow tests, demonstrated good agreement between the ESBGK model and Direct Simulation Monte Carlo (DSMC) results, confirming the physical fidelity of the approach. Finally, the method was successfully applied to a challenging orifice flow with moving boundaries, illustrating its robustness and flexibility for complex rarefied gas dynamics problems. These results indicate that the proposed scheme provides an effective deterministic framework for simulating polyatomic kinetic models across a wide range of flow regimes.

\textbf{Acknowledgments: } This work was supported by the European Union’s Framework Program for Research and Innovation Horizon Europe under the Marie~Skłodowska-Curie Doctoral Networks action (HORIZON-MSCA-2021-DN-01), Grant Agreement No.~101072546 (DATAHYKING). 

\newpage
\printbibliography

\end{document}